\documentclass[extended]{MyStyle}            
\smartqed  
\usepackage{graphicx}
\usepackage{amssymb}
\usepackage{amssymb,amsmath,color,soul,enumerate}
\usepackage{bm}
\usepackage{mathrsfs}
\usepackage{enumerate}
\usepackage[mathscr]{eucal}
\usepackage{eqlist}
\usepackage{graphicx}
\usepackage{booktabs}

\numberwithin{equation}{section}
\numberwithin{proposition}{section}
\numberwithin{theorem}{section}
\numberwithin{definition}{section}
\numberwithin{example}{section}
\numberwithin{lemma}{section}
\numberwithin{corollary}{section}

\newcommand{\de}{\partial}
\newcommand{\ri}{\rightarrow}

\newcommand{\RR}{{\mathbb R}}

\newcommand{\Div}{\mbox{\rm Div\,}}
\newcommand{\bu}{\mbox{\boldmath{$u$}}}
\newcommand{\bv}{\mbox{\boldmath{$v$}}}
\newcommand{\bw}{\mbox{\boldmath{$w$}}}
\newcommand{\bx}{\mbox{\boldmath{$x$}}}

\newcommand{\by}{\mbox{\boldmath{$y$}}}

\newcommand{\bsigma}{\mbox{\boldmath{$\sigma$}}}
\newcommand{\bmu}{\mbox{\boldmath{$\mu$}}}
\newcommand{\btau}{\mbox{\boldmath{$\tau$}}}
\newcommand{\bvarepsilon}{\mbox{\boldmath{$\varepsilon$}}}
\newcommand{\bnu}{\mbox{\boldmath{$\nu$}}}

\newcommand{\bz}{\mbox{\boldmath{$z$}}}

\journalname{}
\begin{document}

\title{Existence results for variational-hemivariational inequality systems with nonlinear couplings 
}

\titlerunning{Existence results for variational-hemivariational inequality systems with nonlinear couplings}        

\author{\normalsize{\rm Yunru} {\sc Bai}$^a$ \and {\rm Nicu\c sor} {\sc Costea}$^b$ \and {\rm Shengda} {\sc Zeng}$^c$}

\authorrunning{Y. Bai\and N. Costea \and S. Zeng} 

\institute{
{\bf Y. Bai} \at
              {{\scriptsize School of Science, Guangxi University of Science and Technology, Liuzhou 545006, Guangxi, China \\
 }
 \email{{\tt yunrubai@163.com}}  }
 \and
{\bf N. Costea} \at
              {{\scriptsize  Department of Mathematics and Computer Science,  National University of Science and Technology  {\sc Politehnica}  Bucharest, 313 Splaiul Independen\c tei, 060042 Bucharest, Romania \\
 }
             \email{{\tt nicusorcostea@yahoo.com;\ nicusor.costea2606@upb.ro}}  }
\and
{\bf S. Zeng} \at
              {{\scriptsize Guangxi Colleges and Universities Key Laboratory of Complex System Optimization and Big Data Processing, Yulin Normal University, Yulin 537000, Guangxi, P.R. China \& Jagiellonian University in Krakow,
Faculty of Mathematics and Computer Science, ul. Lojasiewicza 6, 30-348 Krakow, Poland}\\
             \email{{\tt zengshengda@163.com;\ shengdazeng@gmail.com}}}\\
}

\maketitle

\begin{abstract}  {\footnotesize In this paper we investigate a system of coupled inequalities consisting of a variational-hemivariational inequality and a quasi-hemivariational inequality on Banach spaces. The approach is topological, and a wide variety of existence results is established for both bounded and unbounded constraint sets in real reflexive Banach spaces. Applications to Contact Mechanics are provided in the last section of the paper. More precisely, we consider a contact model with (possibly) multivalued constitutive law whose variational formulation leads to a coupled system of inequalities. The weak solvability of the problem is proved via employing the theoretical results obtained in the previous section. The novelty of our approach comes from the fact that we consider two potential contact zones and the variational formulation allows us to determine simultaneously the displacement field and the Cauchy stress tensor.
}
\end{abstract}

\section{Introduction and preliminaries}

Let $X,{Y}$ be two real reflexive Banach spaces and $K\subseteq X$, $\Lambda\subseteq {Y}$ be nonempty, closed  and convex subsets. Assume $B\colon X\times{Y}\to \RR$ is a  nonlinear functional, $\phi\colon X\to (-\infty,\infty]$ is a proper, convex and lower semicontinuous function such that $D(\phi)\cap K\neq \emptyset$, $Z_1,Z_2$ are Banach spaces and $\gamma_1\colon X\to Z_1$, $\gamma_2\colon Y\to Z_2$ are linear and compact operators, $J\colon Z_1\to\mathbb{R}$ and $L\colon Z_2\to \mathbb{R}$ are locally Lipschitz functionals, $H\colon Y\to (0,\infty)$ is a given function, and $F\colon X\to X^\ast$, $G\colon Y\to Y^\ast$ are prescribed nonlinear operators. In the present, we are interested in the following system of inequalities:

{\it Find $(u,\sigma)\in \left( K\cap D(\phi) \right)\times\Lambda$ such that
\begin{align}\label{eqn1} \left\{
\begin{array}{ll}
B(v,\sigma)-B(u,\sigma)+\phi(v)-\phi(u)+J^0(\gamma_1 (u);\gamma_1 (v-u))\geq \langle F(u),v-u\rangle, \forall v\in K\\
B(u,\mu)-B(u,\sigma)+H(\sigma)L^0(\gamma_2 (\sigma);\gamma_2 (\mu-\sigma))\geq \langle G(\sigma),\mu-\sigma \rangle, \forall \mu\in \Lambda. 
\end{array}
\right.
\end{align}}

\noindent
We point out the fact that the first line of problem (\ref{eqn1}) represents a {\it nonlinear variational-hemivariational inequality}, while  the second line is {\it a quasi-hemivariational inequality}.  Observe that the unknown pair $(u,\sigma)$ appears in both inequalities of (\ref{eqn1}) via the functional $B$ which will be called {\it coupling functional}, so, the system (\ref{eqn1}) is called {\it coupled}.

To our best knowledge, the study of systems consisting of a variational inequality and a functional equations goes back to \cite{Ekeland-Temam99,Has-Hla-Necas96}, whereas coupled systems governed by a hemivariational inequality were first investigated by Matei \cite{Matei19} who established existence results (see also  \cite{Bai-Mig-Zeng20} for conditions ensuring the uniqueness of the solution) for the following problem

{\it Find $(u, \sigma)\in K\times\Lambda$ such that}
\begin{align}\label{eqn2}
\left\{
\begin{array}{ll}
\langle A(u),v-u \rangle+b(v-u,\sigma)+J^0(\gamma (u);\gamma (v)-\gamma (u)) \geq \langle f,v-u\rangle,\ &\forall v\in K,\\
 b(u,\mu-\sigma)\leq 0,\ &\forall \mu\in\Lambda,
\end{array}
\right.
\end{align}
where $X,Y$ are reflexive Banach spaces, $K\subseteq X$ and $\Lambda\subseteq Y$  are nonempty closed convex constraint sets, $Z$ is a Banach space, $f\in X^\ast$, $A\colon X\to X^\ast$ is a nonlinear operator, $J\colon Z\to \mathbb{R}$ is locally Lipschitz, $\gamma\colon X\to Z$ is a linear and continuous operator and $b\colon X\times Y\to \mathbb{R}$ is a bilinear continuous function. 

We point out the fact that system (\ref{eqn2}) can be a  useful mathematical model to solve various complex engineering problems, such as thermo-piezoelectric media problems with semi-permeable thermal boundary conditions, fluid mechanics problems with nonmonotone constitutive law, population dynamic problems,  and so forth, see for instance, \cite{han2017numerical,Matei-RWA14,Matei15,Mig-Bai-Zeng19,zeng2021nonlinear}.

The first paper that we are aware of dealing with coupled systems consisting of two hemivariational inequalities is \cite{Matei22} where the existence of solution for the following problem is studied:

{\it Find $(u, \sigma)\in X\times\Lambda$ such that}
\begin{align}\label{eqn3}
\left\{
\begin{array}{ll}
a(u,v-u)+b(v-u,\sigma)+\phi(v)-\phi(u)+J^0(\gamma_1 (u);\gamma_1 (v)-\gamma (u)) \geq  (f,v-u)_X,\ &\forall v\in X,\\
 b(u,\mu-\sigma) -\psi(\mu)+\psi(\lambda)-L^0(\gamma_2(\lambda);\gamma_2(\mu-\lambda))\leq 0,\ &\forall\mu\in\Lambda,
\end{array}
\right.
\end{align}
where $X,Y$ are Hilbert spaces,  $\Lambda\subseteq Y$  is a nonempty closed convex constraint sets, $Z_1,Z_2$ are Hilbert spaces, $f\in X$, $a: X\times X\to\mathbb{R}$ is a bilinear continuous coercive form, $J\colon Z_1\to \mathbb{R}$ and $L:Z_2\to\mathbb{R}$ are locally Lipschitz, $\gamma_1\colon X\to Z$ and $\gamma_2:Z_2\to \mathbb{R}$ are linear and compact operators and $b\colon X\times Y\to \mathbb{R}$ is a bilinear continuous coupling function.

Note that, if we choose $B(u,\sigma):=b(u,\sigma)$,  $F(u):= f-A(u)$,  $K:=X$, $\phi\equiv0$, $H\equiv0$ and $G\equiv 0$, then the system \eqref{eqn1} reduces to

 {\it Find $(u, \sigma)\in X\times\Lambda$ such that}
\begin{equation}\label{eqn4} \left\{
\begin{array}{l}
\langle A(u),v-u \rangle+b(v-u,\sigma)+J^0(\gamma (u);\gamma (v)-\gamma (u)) \geq \langle f,v-u\rangle,  \ \forall v\in X,\\
 b(u,\mu-\sigma)\geq 0,\ \forall \mu\in \Lambda,
\end{array}
\right.
\end{equation}
whereas for $B(u,\sigma):=b(u,\sigma)+\psi(\sigma)$, $\langle F(u),v\rangle:=(f,v)-a(u,v)$,  $G\equiv 0_Y$, $K:=X$  and $H\equiv 1$ system \eqref{eqn1} becomes 

{\it Find $(u, \sigma)\in X\times\Lambda$ such that}
\begin{align}\label{eqn5}
\left\{
\begin{array}{ll}
a(u,v-u)+b(v-u,\sigma)+\phi(v)-\phi(u)+J^0(\gamma_1 (u);\gamma_1 (v)-\gamma (u)) \geq  (f,v-u)_X,\ &\forall v\in X,\\
 b(u,\mu-\sigma) +\psi(\mu)-\psi(\lambda)+L^0(\gamma_2(\lambda);\gamma_2(\mu-\lambda))\geq 0,\ &\forall\mu\in\Lambda,
\end{array}
\right.
\end{align}

Note that the particular cases \eqref{eqn3} and \eqref{eqn4} are very similar to \eqref{eqn2} and \eqref{eqn3}, respectively,  but the second inequality is reversed, therefore, even for bilinear coupling functions our results are new and supplement the existing knowledge in the literature.

To our best knowledge, systems with nonlinear coupling functions were only recently considered in \cite{Cos23} in the framework of variational inequalities. More precisely, the following problem was investigated:

{\it Find $(u,\sigma)\in  K\times\Lambda$ such that
\begin{align}\label{eqn6} \left\{
\begin{array}{ll}
B(v,\sigma)-B(u,\sigma)+ \chi(u,v-u)\geq \langle f,v-u\rangle, \forall v\in K\\
B(u,\mu)-B(u,\sigma)+\psi(\sigma,\mu-\sigma)\leq \langle g,\mu-\sigma \rangle, \forall \mu\in \Lambda
\end{array}
\right.
\end{align}}
where $\chi:X\times X \to \mathbb{R}, \psi:Y\times Y \to\mathbb{R}$ are nonlinear functionals and $f\in X^\ast$, $g\in Y^\ast$.
Note that the second inequality of \eqref{eqn1} and \eqref{eqn3}, respectively, have opposite sign, therefore even for particular cases the solution sets of the two problems do not coincide. Moreover, besides the fact that hemivariational inequalities are a generalization of variational inequalities further difficulties occur due to the presence of the nonlinear operators $F$ and $G$ in each inequality of \eqref{eqn1} and the fact that the second inequality is of quasi-hemivariational type (since, in general, there does not exist a locally Lipschitz functional $T\colon Y\to\mathbb{R}$ such that $\partial_C T(\sigma)=H(\sigma)\partial_C J^0(\gamma_2 (\sigma))$). Consequently, the novelty of our approach comes on the one hand from the fact that this is the first paper dealing with coupled systems consisting of a variational-hemivariational inequality and a quasi-hemivariational inequality and, on the other hand, that the coupling functional is assumed nonlinear (and even in the bilinear case our system is similar but not equivalent to systems studied in the aforementions papers, as the second inequality is reversed.)

Finally, we point out the fact that choosing opposite sign for the second  inequality of \eqref{eqn1} comes from applications point of view as it allows to tackle  the weak solvability of general 3D contact models with multivalued constitutive law via bipotentials, an approach that has captured special attention, see e.g., \cite{Bod-Sax,Bul-deSax-Vall10,BSV,Cos-Csi-Var15,Sax-Feng98,Mat13,Mat-Nic11}, since this method was introduced at the beginning of 1990's by de Saxc\' e \& Feng  \cite{Sax-Feng91}.  Such an application is provided in the last section of the paper.

\medskip

We fix next basic notations and recall some concepts and  facts which we need in this paper. For more  details and connections one can consult the monographs \cite{Clarke,Cos-Kri-Var,Ekeland-Temam99,sofonea2017variational}.

For a real Banach space $E$ we denote its norm by $\|\cdot\|_E$ and by $E^\ast$ its dual space. The duality pairing between $E^\ast$ and $E$ will be denoted by $\langle \cdot,\cdot \rangle_{E^\ast\times E}$. If there is no danger of confusion the subscripts will be ignored, i.e., we will simply write $\langle\cdot,\cdot\rangle$ and $\|\cdot\|$ from time to time. Particularly, if $E$ is a Hilbert space, then $(\cdot,\cdot)_E$ stands for the inner product of $E$.

Let  $E,F$ be two real Banach spaces. If for each $u\in E$ there exists a corresponding subset $A(u)\subseteq F$, then we say $A\colon E\to 2^{F}$ is a {\it set-valued mapping} (or {\it multifunction}) from $E$ to $F$. For a set-valued mapping $A\colon E\to 2^{F}$ we define {\it domain} of $A$ to be the  set
$$
D(A):=\{x\in E:\ A(x)\neq\emptyset\}.
$$
The {\it range} of $A$ is the set
$$
R(A):=\bigcup_{x\in E}A(x),
$$
and the {\it graph} of $A$ is the set
$$
Gr(A):=\{(x,y)\in E\times F: \  \ y\in A(x) \mbox{ and }x\in D(A)\}.
$$
The {\it inverse} of the set-valued mapping $A\colon E\to 2^F$ is the set-valued mapping $A^{-1}\colon F\to 2^E$ defined by
$$
A^{-1}(y):=\{x\in E: \ \ y\in A(x)\}.
$$

We point out the fact  that every single-valued map $A:E\rightarrow F$ may be  regarded as the set-valued map $E\ni x\mapsto \{A(x)\}\in 2^{F}$. In this case we identify $A(x)$ with its unique element.

\begin{definition}\label{C-ZDissipative}
Let  $\beta\colon E\to\mathbb{R}$ be a given functional. A set-valued mapping $A\colon E\to 2^{E^\ast}$ is called {\it $\beta$-relaxed dissipative} if
\begin{equation}\label{C-ZDiss}
\langle \zeta_y-\zeta_x,y-x\rangle \leq \beta(y-x)\mbox{ for all } (x,\zeta_x),(y,\zeta_y)\in Gr(A).
\end{equation}
\end{definition}

\begin{remark}
If $\beta(x):=-m\|x\|^2$ for some $m>0$, then $A$ is said to be {\it strongly dissipative}, while for $\beta\equiv 0$ the set-valued mapping $A$ is called {\it dissipative}.  If $A$ is single-valued, then \eqref{C-ZDiss} reduces to
$$
\langle A(y)-A(x),y-x\rangle \leq \beta(y-x)\mbox{ for all } x,y\in E.
$$

Moreover, we have
\begin{itemize}
\item[$\bullet$] if $-A$ is $-\beta$-relaxed dissipative, then $A$ is called  {\it $\beta$-relaxed monotone}, i.e.,     \begin{equation*}
\langle \zeta_y-\zeta_x,y-x\rangle \ge \beta(y-x)\mbox{ for all } (x,\zeta_x),(y,\zeta_y)\in Gr(A).
\end{equation*}
\item[$\bullet$]  if $-A$ is strongly dissipative (namely, $A$ is $-\beta$-relaxed dissipative and $\beta(x):=m\|x\|^2$ for some $m>0$), then $A$ is called {\it strongly monotone}, i.e.,
    \begin{equation*}
\langle \zeta_y-\zeta_x,y-x\rangle \ge m\|y-x\|^2\mbox{ for all } (x,\zeta_x),(y,\zeta_y)\in Gr(A).
\end{equation*}
\item[$\bullet$]  if $-A$ is  dissipative (i.e., $\beta\equiv0$), then $A$ is called  {\it monotone}, i.e.,
\begin{equation*}
\langle \zeta_y-\zeta_x,y-x\rangle \ge0\mbox{ for all } (x,\zeta_x),(y,\zeta_y)\in Gr(A).
\end{equation*}
\item[$\bullet$]  if $-A$ is $-\beta$-relaxed dissipative and $\beta(x):=-m\|x\|^2$ for some $m>0$, then $A$ is called {\it relaxed monotone with constant $m$}, i.e.,
    \begin{equation*}
\langle \zeta_y-\zeta_x,y-x\rangle \ge -m\|y-x\|^2\mbox{ for all } (x,\zeta_x),(y,\zeta_y)\in Gr(A).
\end{equation*}
\end{itemize}
\end{remark}

\begin{definition}\label{C-ZHemicont}
A single-valued operator $A\colon E\to E^\ast$ is called {\it hemicontinuous}, if for any $u,v,w\in E$ one has
$$
\lim_{t\to 0_+}\langle A(u+t(v-u)),w\rangle=\langle A(u),w\rangle.
$$
\end{definition}

Next, we recall two kinds of set-valued mappings that play an important role in Nonsmooth Analysis, namely, the {\it subdifferential} of a convex functional and the {\it Clarke subdifferential} of a locally Lipschitz functional.

Let  $\phi\colon E\to (-\infty,+\infty]$ be a prescribed functional. We define  its {\it effective domain} to be  the set
$$
D(\phi):=\{x\in E: \ \ \phi(x)<+\infty\}.
$$

\begin{definition}\label{C-ZFunctionals}
A functional $\phi\colon E\to (-\infty,\infty]$ is said to be:
\begin{description}
\item $\bullet$ {\it proper}, if $D(\phi)\neq \emptyset$;

\item $\bullet$  {\it lower semicontinuous}, if $\liminf_{n\to \infty} \phi(x_n)\geq \phi(x)$, whenever $x_n\to x$;

\item $\bullet$ {\it convex}, if $\phi(t x+(1-t)y)\leq t\phi(x)+(1-t)\phi(y)$, for all $x,y \in E$ and all $t\in [0,1]$;

\item  $\bullet$ {\it G\^ ateaux differentiable} at $x\in D(\phi)$, if there exists an element $ \nabla \phi(x)\in E^\ast$ (called the {\it gradient} of $\phi$ at $x$)  such that
$$
\lim\limits_{t \rightarrow 0_+}\frac{\phi(x+t y)-\phi(x)}{t}=\left\langle \nabla \phi(x),y\right\rangle, \mbox{ for all }y\in E.
$$
\end{description}
\end{definition}

Conversely, we say that $\phi\colon [-\infty,\infty)\to \mathbb{R}$ is {\it upper semicontinuous} (resp. {\it concave}) if $-\phi$ is lower semicontinuous (resp. convex).

The following result characterizes the G\^ ateaux differentiability of convex functionals.
\begin{proposition}\label{Prop1}
Let $\phi\colon E\to \RR$ be a G\^ ateaux differentiable functional. The following statements are equivalent:
\begin{enumerate}[{\rm(i)}]
	\item $\phi$ is convex;
	\item  $\phi(y)-\phi(x)\geq \langle  \nabla\phi(x),y-x\rangle$, for all $y\in E$;
	\item  $\langle \nabla\phi(y)-\nabla\phi(x),y-x\rangle\geq 0$, for all $x,y\in E$.
\end{enumerate}
\end{proposition}
A direct consequence of the above result is that convex and G\^ ateaux differentiable functions are in fact lower semicontinuous. Proposition \ref{Prop1} also suggests the following generalization of the gradient of a convex function.

\begin{definition}\label{Defin1}
Let $\phi:E\rightarrow (-\infty,+\infty]$ be a convex  function. The {\it subdifferential} of $\phi$ at a point $x\in D(\phi)$ is the set
\begin{equation}\label{Eqn2}
\de \phi(x):=\left\{\xi\in E^\ast: \  \ \langle \xi,y-x\rangle \leq \phi(y)-\phi(x)\mbox{ for all } y\in E\right\},
\end{equation}
and $\de\phi(x):=\emptyset$ if $x\not\in D(\phi)$.
\end{definition}

It is well known that if $\phi$ is convex and G\^ ateaux differentiable at a point $x\in {\rm int\;} D(\phi)$, then $\de \phi(x)$ contains exactly one element, namely, $\partial \phi(x)=\{ \nabla \phi(x)\}$. Moreover, it is well-known that the set-valued mapping $\partial \phi\colon E\to 2^{E^*}$ is monotone, i.e., $-\partial \phi$ is dissipative.

The {\it Fenchel conjugate} of $\phi\colon E\to (-\infty,+\infty]$  is the functional $\phi^\ast\colon E^\ast\to (-\infty,+\infty]$ defined by
$$
\phi^\ast(\xi):=\sup\limits_{x\in E} \left\{\langle \xi,x\rangle-\phi(x)\right\}\mbox{ for all $\xi\in E^*$}.
$$

\begin{proposition}\label{prop2}
Let $\phi\colon E\to (-\infty,+\infty]$ be a proper, convex and lower semicontinuous function. Then, the following statements are true:
\begin{enumerate}[{\rm(i)}]
	\item  $\phi^\ast$ is proper, convex and lower semicontinuous;
	\item  $\phi(x)+\phi^\ast(\xi)\geq \langle \xi,x\rangle$, for all $x\in E,\ \xi\in E^\ast$;
	\item $\xi\in \de \phi(x) \Leftrightarrow x\in \de \phi^\ast(\xi)\Leftrightarrow \phi(x)+\phi^\ast(\xi)= \langle \xi,x\rangle$.
\end{enumerate}
\end{proposition}

\begin{definition}\label{Defin2}
A {\it bipotential} is a function $b\colon E\times E^\ast\to (-\infty,+\infty]$ satisfying the following conditions:
\begin{enumerate}[{\rm(i)}]
	\item  for any $x\in E$, if $D(b(x,\cdot))\neq\varnothing$, then $b(x,\cdot)$ is convex and lower semicontinuous, and for any $\xi\in E^\ast$, if $D(b(\cdot,\xi))\neq\varnothing$, then $b(\cdot,\xi)$ is convex  and lower semicontinuous;
	\item   for all $x\in E$ and all $\xi\in E^\ast$ it holds $b(x,\xi)\geq \langle \xi,x\rangle$;
	\item  $\xi\in \de b(\cdot,\xi)(x)\Leftrightarrow x\in \de b(x,\cdot)(\xi) \Leftrightarrow b(x,\xi)=\langle \xi,x\rangle$.
 \end{enumerate}
\end{definition}

We recall that a function $ h \colon E\to \RR$ is called {\it locally Lipschitz} if for every $x\in E$ there exists a neighborhood $U$ of $x$ and a positive constant $c_x=c_x(U)$ such that
$$
| h (z)- h (y)|\leq c_x \|z-y\| \mbox { for all } y,z\in U.
$$

\begin{definition}\label{Defin3}
Let $ h :E\rightarrow\RR$ be a locally Lipschitz function. The {\it generalized derivative} (in the sense of Clarke) of $ h $ at a point $x$ in the direction $y$, denoted by $ h ^0(x;y)$, is defined by
  $$
   h ^0(x;y):=\limsup\limits_{\stackrel{z\rightarrow x}{t \downarrow 0}}\frac{ h (z+t y)- h (z)}{t}.
  $$
\end{definition}

Let $ h \colon E_1\times\ldots\times E_n\to \RR$ be locally Lipschitz with respect to the $k^{th}$ variable. So, in the sequel,  we denote by $ h ^0_{,k}(x_1,\ldots,x_k,\ldots,x_n;y_k)$ the {\it partial generalized derivative} of $ h $ with respect to the $k^{th}$ variable, that is,
$$
 h ^0_{,k}(x_1,\ldots,x_k,\ldots,x_n;y_k):=\limsup\limits_{\stackrel{z_k\rightarrow x_k}{t\downarrow 0}}\frac{ h (x_1,\ldots,z_k+t y_k,\ldots,x_n)- h (x_1,\ldots,z_k,\ldots,x_n)}{t}.
$$
 The next results recall some important properties of the generalized derivative and subgradient in the sense of Clarke, (for the details proof see Clarke \cite{Clarke}).

 \begin{proposition}\label{prop3}
Let $ h \colon E\to\RR$ be a locally Lipschitz function such that $c_x>0$ is the constant near the point $x\in E$. Then, we have
\begin{enumerate}[{\rm(i)}]
	\item  the function $y\mapsto  h ^0(x;y)$ is finite, positively homogeneous, subadditive and satisfies
	$$| h ^0(x;y)|\leq c_x\|y\|\mbox{ for all $x,y\in E$};$$
	\item  the function $(x,y)\mapsto  h ^0(x;y)$ is upper semicontinuous;
	\item  $(- h )^0(x;y)= h ^0(x;-y)$.
\end{enumerate}
\end{proposition}

\begin{definition}\label{Defin4}
The {\it Clarke subdifferential} of a locally Lipschitz function $ h \colon E\to\RR$ at a point $x$, denoted by $\de_C h (x)$, is the following subset of $E^\ast$
$$
\de_C h (x):=\left\{\xi\in E^\ast: \  \ \langle \xi,y \rangle\leq  h ^0(x;y)\mbox{ for all }y\in E\right\}.
$$
\end{definition}
Let $ h \colon E_1\times \ldots \times E_n\to \RR$ be locally Lipschitz with respect to the $k^{th}$ variable. In what follows, we use the similar way to denote the partial Clarke subdifferential of $h$ with respect to the $k^{th}$ variable, thus,
$$
\de_C^k h (x_1,\ldots,x_n):=\left\{ \xi_k\in E_k^\ast: \ \ \langle \xi_k,x_k\rangle\leq  h ^0_{,k}(x_1,\dots,x_k,\ldots,x_n;y_k) \mbox{ for all }y_k\in E_k\right\}.
$$

\begin{proposition}\label{CCV-Prop2}
Let $h\colon E\to \RR$ be a locally Lipschitz functional. Then the following properties hold:
\begin{enumerate}[{\rm(i)}]
  \item For each $x\in E$, $\partial_C h(x)$ is a nonempty, convex, weak$^\ast$-compact subset of $E^\ast$ and
  $$
  \|\zeta\|_{E^*}\leq c_x, \mbox{ for all }\zeta\in \partial_C h(x),
  $$
  where $c_x$ is the Lipschitz constant near $x$;
  \item  For each $y\in E$ one has $h^0(x;y)=\max\limits_{\zeta \in \partial_C h(u)} \langle \zeta,y\rangle$;
  \item The set-valued mapping $x\mapsto \partial_C h(x)$ is weakly$^\ast$-closed;
  \item The set-valued mapping $x\mapsto \partial_C h(x)$ is upper semicontinuous from $E$ into $E^\ast$ endowed with the weak$^\ast$-topology.
 \end{enumerate}
\end{proposition}

We end this section with an alternative theorem for equilibrium problems due to Mosco \cite{mosco} which will play a center role in  proving the existence of solutions for the coupled system of inequalities (\ref{eqn1}).
\begin{theorem}[Mosco \cite{mosco}]\label{Mosco}
Let $\mathcal{K}$ be a nonempty, compact and convex subset of a topological
vector space $E$ and $\varphi:E\ri (-\infty,\infty]$ be a proper,
convex and lower semicontinuous functional such that ${
D}(\varphi)\cap \mathcal{K}\neq\emptyset$. Assume $g,h\colon E\times E\to\RR$ are
two functions that satisfy the following conditions:
\begin{enumerate}[{\rm(a)}]
\item  $g(x,y)\geq h(x,y) \mbox{ for all } x,y\in E$;

\item  $y\mapsto h(x,y)$ is  convex for each $y\in E$;

\item  $x\mapsto g(x,y)$ is upper semicontinous for
each $x\in E$.
\end{enumerate}
Then, for each $a\in\mathbb{R}$  the following alternative holds:
\begin{description}
\item $\bullet$ either there exists $x_0\in { D}(\varphi)\cap \mathcal{K}$ such that \begin{eqnarray*}
    g(x_0,y)+\varphi(y)-\varphi(x_0)\geq a\mbox{ for all } y\in E,
    \end{eqnarray*}
\item  $\bullet$ or there exists $y_0\in E$ such that $h(x_0,x_0)<a$.
\end{description}
\end{theorem}

\section{Existence results for bounded and unbounded constraint sets}

The primary goal of this section is to establish the existence of solutions for the nonlinear coupled system \eqref{eqn1}. Our approach is topological. More precisely we employ the alternative Theorem~\ref{Mosco} for general equilibrium problems to obtain the existence theorems. The first set of assumptions under which we are able to establish the existence of solutions is listed below.

\begin{description}
\item ${\bf (H_0)}$ The functionals $\phi, J, L, H$ and the operators $\gamma_1,\gamma_2$ satisfy the following conditions:

\begin{enumerate}[{\rm(i)}]
\item  $\phi\colon X\to(-\infty,\infty]$ is a proper, convex and lower semicontinuous functional such that $D(\phi)\cap K\neq\varnothing$;

\item  $J\colon Z_1\to \mathbb{R}$ and $L\colon Z_2\to \mathbb{R}$ are locally Lipschitz functionals;

\item  $\gamma_1\colon X\to Z_1$ and $\gamma_2\colon Y\to Z_2$ are linear and compact operators;

\item $H\colon Y \to (0,m_H)$ is weakly continuous, i.e., $H(\sigma_\alpha)\to H(\sigma)$ whenever $\sigma_\alpha\rightharpoonup \sigma$ in $Y$, with some $m_H>0$.

\end{enumerate}

\item ${\bf (H_B)}$ $B\colon X\times Y\to \mathbb{\mathbb{R}}$ is convex and lower semicontinuous in $X\times Y$.

\item ${\bf (H_F^1)}$ $F\colon X\to X^\ast$ is a nonlinear operator such that  the mapping $u\mapsto \langle F(u) ,v-u\rangle $ is weakly lower semicontinuous for each $v\in X$, i.e.,
$$
\liminf_{n\to\infty}\langle F(u_n) ,v-u_n\rangle\ge\langle F(u) ,v-u\rangle,
$$
   whenever $u_n\rightharpoonup u$ in $X$ as $n\to\infty$ and $v\in X$.

\item ${\bf (H_G^1)}$ $G\colon Y\to Y^\ast$ is a nonlinear operator such that  the mapping $\sigma\mapsto \langle G(\sigma), \mu-\sigma\rangle $ is weakly lower semicontinuous for each $\mu\in Y$, i.e.,
    \begin{eqnarray*}
    \liminf_{n\to\infty}\langle  G(\sigma_n), \mu-\sigma_n\rangle\ge\langle  G(\sigma), \mu-\sigma\rangle,
    \end{eqnarray*}
    whenever $\sigma_n\rightharpoonup \sigma$ in $Y$ as $n\to\infty$ and $\mu\in Y$.

\end{description}

Note that, if  $J\colon Z_1\to \mathbb{R}$ is Lipschitz continuous, then ${\bf (H_0)}$(ii) is automatically fulfilled as $c_x$ is the same for every $x\in Z_1$.
Moreover, ${\bf (H_F^1)}$ is automatically fulfilled if $F(u)\equiv f$ for some fixed $f\in X^\ast$.

\begin{lemma}\label{BoundedCase1}
Assume $X$ and $Y$ are real reflexive Banach spaces and $K\subset X$ and $\Lambda\subset Y$ are nonempty, bounded, closed and convex subsets. If ${\bf (H_0), \; (H_B), \; (H_F^1)} $ and ${\bf (H_G^1)}$ hold, then problem (\ref{eqn1}) possesses at least one solution.
\end{lemma}

\begin{proof}
Let $E:=X\times Y$, $\mathcal{K}:=K\times \Lambda$ and define $\varphi\colon \mathcal{K}\to(-\infty,\infty]$ by
$$
\varphi(u,\sigma):=\phi(u)+I_K(u)+I_\Lambda(\sigma),
$$
where $I_K$ and $I_\Lambda$ are the {\it indicator functions} of $K$ and $\Lambda$, respectively, i.e.,
$$
I_K(u):=\left\{
\begin{array}{ll}
0, &\mbox{ if }u\in K\\
+\infty, & \mbox{ otherwise}
\end{array}
\right.,\mbox{ and }
I_\Lambda(\sigma):=\left\{
\begin{array}{ll}
0, &\mbox{ if }\sigma\in \Lambda\\
+\infty, & \mbox{ otherwise}
\end{array}
\right..
$$
Recall that $K$ and $\Lambda$ are nonempty, bounded, closed and convex subsets of real reflexive Banach spaces $X$ and $Y$, respectively, it follows at once that $\mathcal{K}$ is weakly  compact, while $\varphi$ is proper, convex and lower semicontinuous with $D(\varphi)=(K\cap D(\phi))\times\Lambda$.

Let us consider the bifunction $g\colon \mathcal{K}\times \mathcal{K}\to\mathbb{R}$ by
\begin{align*}
g((u,\sigma) ,(v,\mu)):= & B(v,\mu)-B(u,\sigma)+J^0(\gamma_1(u);\gamma_1(v-u))- \langle F(u),v-u\rangle\\
&+H(\sigma)L^0(\gamma_2(\sigma);\gamma_2(\mu-\sigma))- \langle G(\sigma),\mu-\sigma \rangle.
\end{align*}
 It is not difficult to observe that
 $g((u,\sigma),(u,\sigma))=0$ for all $(u,\sigma)\in \mathcal{K}$. Let $t\in(0,1)$ and $(v_1,\mu_1),(v_2,\mu_2)\in \mathcal K$ be arbitrary and $v_t=tv_1+(1-t)v_2$ and $\mu_t=t\mu_1+(1-t)\mu_2$. From the convexity of $\mathcal K\ni (v,\mu)\mapsto B(v,\mu)\in \mathbb R$ and the positive homogeneity and subadditivity of $v\mapsto J^0(\gamma_1(u);\gamma_1(v))$ and  $\mu\mapsto L^0(\gamma_2(\sigma);\gamma_1(\mu))$, we have
 \begin{align*}
 g((u,\sigma)& ,(v_t,\mu_t))= B(v_t,\mu_t)-B(u,\sigma)+J^0(\gamma_1(u);\gamma_1(v_t-u))- \langle F(u),v_t-u\rangle\\
&+H(\sigma)L^0(\gamma_2(\sigma);\gamma_2(\mu_t-\sigma))- \langle G(\sigma),\mu_t-\sigma \rangle\\
\le &tB(v_1,\mu_1)+(1-t)B(v_2,\mu_2)+J^0(\gamma_1(u);t\gamma_1(v_1-u)+(1-t)\gamma_1(v_2-u))\\
&+H(\sigma)L^0(\gamma_2(\sigma);t\gamma_2(\mu_1-\sigma)+(1-t)\gamma_2(\mu_2-\sigma))-B(u,\sigma)\\
&- \langle F(u),t(v_1-u)+(1-t)(v_2-u)\rangle- \langle G(\sigma),t(\mu_1-\sigma)+(1-t)(\mu_2-\sigma) \rangle\\
\le&t\bigg[B(v_1,\mu_1)+J^0(\gamma_1(u);\gamma_1(v_1-u)) +H(\sigma)L^0(\gamma_2(\sigma);\gamma_2(\mu_1-\sigma))-B(u,\sigma)\\
&- \langle F(u),v_1-u\rangle- \langle G(\sigma),\mu_1-\sigma\rangle\bigg] +(1-t)\bigg[B(v_2,\mu_2)+J^0(\gamma_1(u);\gamma_1(v_2-u)) \\ &+H(\sigma)L^0(\gamma_2(\sigma);\gamma_2(\mu_2-\sigma))- \langle F(u),v_2-u\rangle- \langle G(\sigma),\mu_2-\sigma\rangle-B(u,\sigma)\bigg]\\
\le&tg((u,\sigma) ,(v_1,\mu_1))+(1-t)g((u,\sigma) ,(v_2,\mu_2)).
 \end{align*}
 This means that the mapping $(v,\mu)\mapsto g((u,\sigma),(v,\mu))$ is convex.

We claim that for each $(v,\mu)\in\mathcal K$ the mapping   $\mathcal K\ni (u,\sigma)\mapsto g((u,\sigma),(v,\mu))\in \mathbb R$ is weakly upper semicontinuous. Let sequence $\{(u_n,\sigma_n)\}\subset \mathcal K$ and $(u,\sigma)\in \mathcal K$ be such that
$$
(u_n,\sigma_n)\rightharpoonup(u,\sigma)\mbox{ in $X\times Y $ as $n\to\infty$.}
$$
It follows from the weakly lower semicontinuity of $\mathcal K\ni (u,\sigma)\mapsto B(u,\sigma)\in\mathbb R$, the upper semicontinuity of $Z_1\times Z_1\ni (u,v)\mapsto J^0(u;v)\in\mathbb R$ (see Proposition \ref{CCV-Prop2}) that
\begin{align}\label{eqns3.1}
\left\{\begin{array}{ll}
\liminf_{n\to\infty}B(u_n,\sigma_n)\ge B(u,\sigma),\\ \limsup_{n\to\infty}J^0(\gamma_1(u_n);\gamma_1(v-u_n))\le J^0(\gamma_1(u);\gamma_1(v-u)),
\end{array}\right.
\end{align}
where we have used the compactness of $\gamma_1$. Using hypotheses  ${\bf (H_F^1)} $ and ${\bf (H_G^1)}$, it yields
\begin{eqnarray}\label{eqns3.2}
    \liminf_{n\to\infty}\langle F(u_n) ,v-u_n\rangle\ge\langle F(u) ,v-u\rangle
    \end{eqnarray}
and
\begin{eqnarray}\label{eqns3.3}
\liminf_{n\to\infty}\langle  G(\sigma_n), \mu-\sigma_n\rangle\ge\langle  G(\sigma), \mu-\sigma\rangle.
\end{eqnarray}
Employing hypotheses ${\bf (H_0)}$(iii)-(iv) and the upper semicontinuity of the mapping  $Z_2\times Z_2\ni (\sigma,\mu)\mapsto L^0(\sigma;\mu)\in\mathbb R$ (see Proposition \ref{CCV-Prop2}), one has
\begin{align}\label{eqns3.4}
\nonumber \limsup_{n\to\infty}H(\sigma_n)L^0(\gamma_2(\sigma_n);\gamma_2(\mu-\sigma_n))
\le&
\limsup_{n\to\infty} \underbrace{|H(\sigma_n)-H(\sigma)|}_{\to 0} \underbrace{L^0(\gamma_2(\sigma_n);\gamma_2(\mu-\sigma_n))}_{bounded}\nonumber\\
&+\limsup_{n\to\infty}H(\sigma)L^0(\gamma_2(\sigma_n);\gamma_2(\mu-\sigma_n))\nonumber\\
\le&
\nonumber\limsup_{n\to\infty}H(\sigma)L^0(\gamma_2(\sigma_n);\gamma_2(\mu-\sigma_n))\nonumber\\
\le&H(\sigma)L^0(\gamma_2(\sigma);\gamma_2(\mu-\sigma)),
\end{align}
where we have applied the compactness of $\gamma_2$ and the following estimates
\begin{align*}
|L^0(\gamma_2(\sigma_n); \gamma_2(\mu-\sigma_n))|& =\left|\max_{\zeta\in\partial_C J(\gamma_2(\sigma_n))}\langle \zeta,\gamma_2(\mu-\sigma_n)\rangle\right|=|\langle \zeta_n,\gamma_2(\mu)-\gamma_2(\sigma_n)\rangle|\\
&\leq \|\zeta_n\|_{Z_2^*} \|\gamma_2(\mu)-\gamma_2(\sigma_n)\|\leq c_\mu \|\gamma_2(\mu)-\gamma_2(\sigma_n)\|.
\end{align*}
The last inequality is obtained by using the fact that $\gamma_2(\sigma_n)\to \gamma_2(\sigma)$ as $n\to\infty$ and Proposition~\ref{prop3}, and $c_\sigma>0$ is the Lipschitz constant of $L$ near the point $\gamma_2\sigma\in Z_2$.
Taking into account (\ref{eqns3.1})--(\ref{eqns3.4}), we infer that
\begin{eqnarray*}
\limsup_{n\to\infty}g((u_n,\sigma_n),(v,\mu))\le g((u,\sigma),(v,\mu)),
\end{eqnarray*}
namely, for each $(v,\mu)\in\mathcal K$ the mapping   $\mathcal K\ni (u,\sigma)\mapsto g((u,\sigma),(v,\mu))\in \mathbb R$ is weakly upper semicontinuous.

Therefore, all conditions of Theorem \ref{Mosco} are verified. So, we are now in position to apply this theorem with $a:=0$, $h:=g$ and $E$ endowed with the weak topology to conclude that there exists an element
$(u_0,\sigma_0)\in \mathcal{K}\cap D(\varphi)$ such that
\begin{equation}\label{BoundedSol1}
g((u_0,\sigma_0),(v,\mu))+\varphi(v,\mu)-\varphi(u_0,\sigma_0)\geq 0\mbox{ for all } (v,\mu)\in E.
\end{equation}
Choosing $\mu:=\sigma_0$ in \eqref{BoundedSol1}, we get that
$$
B(v,\sigma_0)-B(u_0,\sigma_0)+J^0(\gamma_1(u_0);\gamma_1(v-u_0))-\langle F(u_0),v-u_0\rangle +\phi(v)-\phi(u_0)\geq 0
$$
for all $v\in K$. Whereas, if we take
 $v:=u_0$ in \eqref{BoundedSol1}, then
$$
B(u_0,\mu)-B(u_0,\sigma_0)+H(\sigma_0)L^0(\gamma_2(\sigma_0);\gamma_2(\mu-\sigma_0))-\langle G(\sigma_0),\mu-\sigma_0\rangle\geq 0
$$
for all  $\mu\in \Lambda$.
This shows that $(u_0,\sigma_0)$ is a solution of problem (\ref{eqn1}).
\qed
\end{proof}

In hypotheses ${\bf (H_F^1)}$ and ${\bf (H_G^1)}$, we require that  $u\mapsto \langle F(u) ,v-u\rangle $  and $\sigma\mapsto \langle G(\sigma) ,\mu-\sigma\rangle $ are both weakly lower semicontinuous. However, this assumption strictly limits the application of our result, Lemma~\ref{BoundedCase1}. In order to extend the scope of applications to problem (\ref{eqn1}), we, further, make necessary changes to the assumptions on $F$ and $G$, respectively. Under these changes, we obtain a wide variety of existence results to problem (\ref{eqn1}).

To this end, we impose the following assumptions to functions $F$ and $G$.
\begin{description}
\item ${\bf (H_F^2)}$ $F\colon X\to X^\ast$ is hemicontinuous and there exists a concave functional $\beta_F:X\to\mathbb{R}$ with $\beta_F(0)\leq 0$  such that  $F$ is  $\beta_F$-relaxed  dissipative.

\item ${\bf (H_G^2)}$ $G\colon Y\to Y^\ast$ is hemicontinuous and there exists a concave functional $\beta_G\colon Y\to\mathbb{R}$  with $\beta_G(0)\leq 0$ such that $G$ is $\beta_G$-relaxed dissipative.

\item ${\bf (H_F^3)}$ $F\colon X\to X^\ast$ is  hemicontinuous and there exists a concave functional $\beta_T\colon X\to\mathbb{R}$ with $\beta_T(0)\leq 0$  such that  the set-valued operator $T\colon X\to 2^{X^\ast}$ defined by
 $$
 T(u):=F(u)-\gamma_1^\ast \partial_C J(\gamma_1(u))
 $$
 is $\beta_T$-relaxed dissipative.

\item ${\bf (H_G^3)}$ $G\colon Y\to Y^\ast$ is hemicontinuous and there exists a concave functional $\beta_S\colon Y\to\mathbb{R}$ with $\beta_S(0)\leq 0$  such that the set-valued operator $S\colon Y\to 2^{Y^\ast}$ defined by
$$
S(\sigma):=G(\sigma)-H(\sigma)\gamma_2^\ast \partial_C L(\gamma_2(\sigma))
$$
is $\beta_S$-relaxed  dissipative.
\end{description}

The second result regarding the existence of solutions for system \eqref{eqn1} in the case of bounded constraint sets is stated below.

\begin{lemma}\label{BoundedSol2}
The conclusion of Lemma \ref{BoundedCase1} still holds if we replace ${\bf (H_F^1)}$ by ${\bf (H_F^j)}$ or  ${\bf (H_G^1)}$ by ${\bf (H_G^k)}$, with $j,k\in\{1,2,3\}$.
\end{lemma}

\begin{proof}
As in the proof of Lemma \ref{BoundedCase1}, let  $E:=X\times Y$, $\mathcal{K}:=K\times \Lambda$ and $\varphi\colon\mathcal{K}\to (-\infty,\infty]$ be defined by
$$
\varphi(u,\sigma):=\phi(u)+I_K(u)+I_\Lambda(\sigma).
$$

\begin{description}
\item {\sc Case 1.} {\it Assume ${\bf (H_F^2)}$ and ${\bf (H_G^3)}$ hold.}

Let us define the functions $g,h\colon \mathcal{K}\times\mathcal{K}\to \mathbb{R}$ by \begin{align*}
g((u,\sigma),(v,\mu)):=& B(v,\mu)-B(u,\sigma)+J^0(\gamma_1(u);\gamma_1(v-u))-\langle F(v),v-u\rangle \\
 & + H(\mu)L^0(\gamma_2(\mu);\gamma_2(\mu-\sigma))-\langle G(\mu),\mu-\sigma\rangle
\end{align*}
and
\begin{align*}
h((u&,\sigma),(v,\mu)):= B(v,\mu)-B(u,\sigma)+J^0(\gamma_1(u);\gamma_1(v-u))-\langle F(u),v-u\rangle \\
 & -\beta_F(v-u)+ H(\sigma)L^0(\gamma_2(\sigma);\gamma_2(\mu-\sigma))-\langle G(\sigma),\mu-\sigma\rangle -\beta_S(\mu-\sigma).
\end{align*}
Note that $h((u,\sigma),(u,\sigma))=-\beta_F(0)-\beta_S(0)\geq 0$ for all $(u,\sigma)\in\mathcal{K}$. Arguing as in the proof of Lemma~\ref{BoundedCase1}, it is readily seen that $\mathcal K\ni (u,\sigma)\mapsto g((u,\sigma),(v,\mu))\in \mathbb R$ is weakly upper semicontinuous and $\mathcal K\ni (v,\mu)\mapsto h((u,\sigma),(v,\mu))\in\mathbb R$ is convex.  Moreover, hypothesis ${\bf (H_G^3)}$ and Proposition \ref{CCV-Prop2} point out that there exists $\xi_\sigma\in\partial_C L(\gamma_2(\sigma))$ and $\xi_\mu\in\partial_C L(\gamma_2(\mu))$ such that
$$
L^0(\gamma_2(\sigma);\gamma_2(\mu-\sigma))=\langle \xi_\sigma, \gamma_2(\mu-\sigma)\rangle_{Z_2^\ast\times Z_2}=\langle\gamma_2^\ast (\xi_\sigma),\mu-\sigma\rangle_{Y^\ast\times Y},
$$
and
$$
L^0(\gamma_2(\mu);\gamma_2(\mu-\sigma))=\langle \xi_\mu, \gamma_2(\mu-\sigma)\rangle_{Z_2^\ast\times Z_2}=\langle\gamma_2^\ast( \xi_\mu),\mu-\sigma\rangle_{Y^\ast\times Y}.
$$
Thus,
\begin{align*}
g((u,\sigma),(v,\mu)) -h((u,&\sigma) ,(v,\mu))=\beta_F(v-u) -\langle F(v)-F(u),v-u\rangle +\beta_S(\mu-\sigma)\\
& -\langle G(\mu)-H(\mu)\gamma_2^\ast(\xi_\mu)-G(\sigma)+H(\sigma)\gamma_2^\ast(\xi_\sigma) ,\mu-\sigma \rangle\geq 0.
\end{align*}

We apply again Theorem \ref{Mosco} with $a:=0$ to get the existence of $(u_0,\sigma_0)\in \mathcal{K}\cap D(\varphi)$ such that
\begin{equation}\label{BoundedSol2s}
g((u_0,\sigma_0),(w,\lambda))+\varphi(w,\lambda)-\varphi(u_0,\sigma_0)\geq 0\mbox{ for all } (w,\lambda)\in E.
\end{equation}
Let $(v,\mu)\in \mathcal{K}$ be fixed. Then for any $t\in (0,1)$, by inserting $(w,\lambda):=(u_0+t(v-u_0),\sigma_0)$  in \eqref{BoundedSol2s} we get
\begin{eqnarray*}
0&\leq& B(u_0+t(v-u_0),\sigma_0)-B(u_0,\sigma_0)+  J^0(\gamma_1(u_0);t(\gamma_1(v-u_0)))\\
&&-\langle F(u_0+t(v-u_0)),t(v-u_0)\rangle +\phi(u_0+t(v-u_0))-\phi(u_0)\\
&\leq& t\left[ B(v,\sigma_0)-B(u_0,\sigma_0) +J^0(\gamma_1(u_0);\gamma_1(v-u_0))\right.\\
&& \left. -\langle F(u_0+t(v-u_0)),v-u_0\rangle+\phi(v)-\phi(u_0) \right],
\end{eqnarray*}
where we have utilized the convexity of $B$ and $\phi$ as well as the positivity homogeneity and subadditivity of $v\mapsto J^0(u;v)$.  Dividing by $t>0$ and then letting $t\to 0_+$, we apply the hemicontinuity of $F$ to get the first inequality of (\ref{eqn1}). In order to get the second inequality of (\ref{eqn1}), we fix $t\in (0,1)$ and choose $(w,\lambda):=(u_0,\sigma_0+t(\mu-\sigma_0))$ in \eqref{BoundedSol2s} to get
\begin{eqnarray*}
0&\leq& B(u_0,\sigma_0+t(\mu-\sigma_0))-B(u_0,\sigma_0)-\langle G(\sigma_0+t(\mu-\sigma_0)),t(\mu-\sigma_0)\rangle  \\
&&+H(\sigma_0+t(\mu-\sigma_0))L^0(\gamma_2(\sigma_0+t(\mu-\sigma_0));t\gamma_2(\mu-\sigma_0))\\
&\leq& t\left[ B(u_0,\mu)-B(u_0,\sigma_0) -\langle G(\sigma_0+t(\mu-\sigma_0)),\mu-\sigma_0\rangle \right.\\
&& \left. +H(\sigma_0+t(\mu-\sigma_0))L^0(\gamma_2(\sigma_0+t(\mu-\sigma_0));\gamma_2(\mu-\sigma_0)) \right].
\end{eqnarray*}
Dividing again by $t>0$, then taking the $\limsup$ as $t\to 0_+$ we get that $(u_0,\sigma_0)$ also satisfies the second inequality of  (\ref{eqn1}).

\item {\sc Case 2.} {\it Assume that ${\bf (H_F^2)}$ and ${\bf (H_G^2)}$ hold.}

We define the functions $g,h\colon \mathcal{K}\times\mathcal{K}\to \mathbb{R}$ by
\begin{align*}
g((u,\sigma),(v,\mu)):=& B(v,\mu)-B(u,\sigma)+J^0(\gamma_1(u);\gamma_1(v-u))-\langle F(v),v-u\rangle \\
 & + H(\sigma)L^0(\gamma_2(\sigma);\gamma_2(\mu-\sigma))-\langle G(\mu),\mu-\sigma\rangle
\end{align*}
and
\begin{align*}
h((u,\sigma),(v,\mu)):=& B(v,\mu)-B(u,\sigma)+J^0(\gamma_1(u);\gamma_1(v-u))-\langle F(u),v-u\rangle-\beta_F(v-u) \\
 & + H(\sigma)L^0(\gamma_2(\sigma);\gamma_2(\mu-\sigma))-\langle G(\sigma),\mu-\sigma\rangle -\beta_G(\mu-\sigma).
\end{align*}
It is obvious that  $h((u,\sigma),(u,\sigma))=-\beta_F(0)-\beta_S(0)\geq 0$ for all $(u,\sigma)\in\mathcal{K}$. Moreover, the mapping  $\mathcal K\ni (u,\sigma)\mapsto g((u,\sigma),(v,\mu))\in \mathbb R$ is weakly upper semicontinuous and $\mathcal K\ni (v,\mu)\mapsto h((u,\sigma),(v,\mu))\in\mathbb R$ is convex. Additionally, we apply hypotheses ${\bf (H_F^2)}$ and ${\bf (H_G^2)}$ to get
\begin{align*}
g((u,\sigma),(v,\mu)) -h((u,\sigma) ,(v,\mu))=&\beta_F(v-u) -\langle F(v)-F(u),v-u\rangle +\beta_G(\mu-\sigma)\\
& -\langle G(\mu)-G(\sigma) ,\mu-\sigma \rangle\\
\geq& 0.
\end{align*}

This allows us to employ Theorem \ref{Mosco} with $a:=0$ to find $(u_0,\sigma_0)\in \mathcal{K}\cap D(\varphi)$ such that (\ref{BoundedSol2s}) holds.
Let $(v,\mu)\in \mathcal{K}$ be fixed. Inserting $(w,\lambda):=(u_0+t(v-u_0),\sigma_0)$ with $t\in(0,1)$ in \eqref{BoundedSol2s}  we get
\begin{eqnarray*}
0&\leq&  t\left[ B(v,\sigma_0)-B(u_0,\sigma_0) +J^0(\gamma_1(u_0);\gamma_1(v-u_0))\right.\\
&& \left. -\langle F(u_0+t(v-u_0)),v-u_0\rangle+\phi(v)-\phi(u_0) \right].
\end{eqnarray*}
Hence, from the hemicontinuity of $F$, we obtain the first inequality of (\ref{eqn1}). Let $t\in (0,1)$ be arbitrary and put $(w,\lambda):=(u_0,\sigma_0+t(\mu-\sigma_0))$  in \eqref{BoundedSol2s}. We have
\begin{eqnarray*}
0&\leq& B(u_0,\sigma_0+t(\mu-\sigma_0))-B(u_0,\sigma_0)-\langle G(\sigma_0+t(\mu-\sigma_0)),t(\mu-\sigma_0)\rangle  \\
&&+H(\sigma_0)L^0(\gamma_2(\sigma_0);t(\gamma_2(\mu-\sigma_0)))\\
&\leq& t\left[ B(u_0,\mu)-B(u_0,\sigma_0) -\langle G(\sigma_0+t(\mu-\sigma_0)),\mu-\sigma_0\rangle +H(\sigma_0)L^0(\gamma_2(\sigma_0);\gamma_2(\mu-\sigma_0)) \right].
\end{eqnarray*}
Invoking the hemicontinuity of $G$, we obtain the second inequality of  (\ref{eqn1}).

\item{\sc Case 3.} {\it Suppose that ${\bf (H_F^3)}$ and ${\bf (H_G^2)}$ are satisfied.}

Let $g,h\colon \mathcal{K}\times\mathcal{K}\to \mathbb{R}$  be defined by
\begin{align*}
g((u,\sigma),(v,\mu)):=& B(v,\mu)-B(u,\sigma)+J^0(\gamma_1(v);\gamma_1(v-u))-\langle F(v),v-u\rangle \\
 & + H(\sigma)L^0(\gamma_2(\sigma);\gamma_2(\mu-\sigma))-\langle G(\mu),\mu-\sigma\rangle
\end{align*}
and
\begin{align*}
h((u,\sigma),(v,\mu)):=& B(v,\mu)-B(u,\sigma)+J^0(\gamma_1(u);\gamma_1(v-u))-\langle F(u),v-u\rangle-\beta_T(v-u) \\
 & + H(\sigma)L^0(\gamma_2(\sigma);\gamma_2(\mu-\sigma))-\langle G(\sigma),\mu-\sigma\rangle -\beta_G(\mu-\sigma).
\end{align*}

It is obvious that  $h((u,\sigma),(u,\sigma))=-\beta_T(0)-\beta_S(0)\geq 0$ for all $(u,\sigma)\in\mathcal{K}$. Again,  $\mathcal K\ni (u,\sigma)\mapsto g((u,\sigma),(v,\mu))\in \mathbb R$ is weakly upper semicontinuous and $\mathcal K\ni (v,\mu)\mapsto h((u,\sigma),(v,\mu))\in\mathbb R$ is convex. Additionally, we use hypotheses ${\bf (H_F^3)}$ and ${\bf (H_G^2)}$ to get
\begin{align*}
g((u,\sigma),(v,\mu)) -h((u,\sigma) ,(v,\mu))=&\beta_T(v-u) -\langle F(v)-\gamma_1^*\xi_u-F(u)+\gamma_1^*\xi_v,v-u\rangle\\
& +\beta_G(\mu-\sigma) -\langle G(\mu)-G(\sigma) ,\mu-\sigma \rangle\\
\geq& 0,
\end{align*}
where $\xi_u\in \partial_C J(\gamma_1u)$ and $\xi_v\in \partial J_C(\gamma_1v)$ are such that
$$
J^0(\gamma_1(u);\gamma_1(v-u))=\langle \xi_u, \gamma_1(v-u)\rangle_{Z_1^\ast\times Z_1}=\langle\gamma_1^\ast( \xi_u),u-v\rangle_{X^\ast\times X},
$$
and
$$
J^0(\gamma_1(v);\gamma_1(v-u))=\langle \xi_v, \gamma_1(v-u)\rangle_{Z_1^\ast\times Z_1}=\langle\gamma_1^\ast( \xi_v),u-v\rangle_{X^\ast\times X}.
$$

We are now in a position to apply Theorem \ref{Mosco} with $a:=0$ to deduce that there exists   an element $(u_0,\sigma_0)\in \mathcal{K}\cap D(\varphi)$ such that (\ref{BoundedSol2s}) holds.
Let $(v,\mu)\in \mathcal{K}$ be fixed. Then for any $t\in (0,1)$, we insert $(w,\lambda):=(u_0+t(v-u_0),\sigma_0)$  in \eqref{BoundedSol2s} to get
\begin{eqnarray*}
0&\leq&  t\left[ B(v,\sigma_0)-B(u_0,\sigma_0) +J^0(\gamma_1(u_0+t(v-u_0));\gamma_1(v-u_0))\right.\\
&& \left. -\langle F(u_0+t(v-u_0)),v-u_0\rangle+\phi(v)-\phi(u_0) \right].
\end{eqnarray*}
Hence, from the hemicontinuity of $F$ and upper semicontinuity of $Z_1\ni u\mapsto J^0(u;v)$, we obtain the first inequality of (\ref{eqn1}). For $t\in (0,1)$ plugging  $(w,\lambda):=(u_0,\sigma_0+t(\mu-\sigma_0))$  in \eqref{BoundedSol2s} we have
\begin{eqnarray*}
0&\leq& t\left[ B(u_0,\mu)-B(u_0,\sigma_0) -\langle G(\sigma_0+t(\mu-\sigma_0)),\mu-\sigma_0\rangle +H(\sigma_0)L^0(\gamma_2(\sigma_0);\gamma_2(\mu-\sigma_0)) \right].
\end{eqnarray*}
Invoking the hemicontinuity of $G$ the second inequality of  (\ref{eqn1}) is obtained.

\item {\sc Case 4.} {\it Let ${\bf (H_F^3)}$ and ${\bf (H_G^3)}$ be fulfilled.}

Consider the functions $g,h\colon \mathcal{K}\times\mathcal{K}\to \mathbb{R}$ defined by
\begin{align*}
g((u,\sigma),(v,\mu)):=& B(v,\mu)-B(u,\sigma)+J^0(\gamma_1(v);\gamma_1(v-u))-\langle F(v),v-u\rangle \\
 & + H(\mu)L^0(\gamma_2(\mu);\gamma_2(\mu-\sigma))-\langle G(\mu),\mu-\sigma\rangle
\end{align*}
and
\begin{align*}
h((u,\sigma),(v,\mu)):=& B(v,\mu)-B(u,\sigma)+J^0(\gamma_1(u);\gamma_1(v-u))-\langle F(u),v-u\rangle-\beta_T(v-u) \\
 & + H(\sigma)L^0(\gamma_2(\sigma);\gamma_2(\mu-\sigma))-\langle G(\sigma),\mu-\sigma\rangle -\beta_S(\mu-\sigma).
\end{align*}
Then, we have  $h((u,\sigma),(u,\sigma))=-\beta_T(0)-\beta_S(0)\geq 0$ for all $(u,\sigma)\in\mathcal{K}$, the  mapping $\mathcal K\ni (u,\sigma)\mapsto g((u,\sigma),(v,\mu))\in \mathbb R$ is weakly upper semicontinuous, while the mapping $\mathcal K\ni (v,\mu)\mapsto h((u,\sigma),(v,\mu))\in\mathbb R$ is convex. Let $\xi_\sigma\in\partial_C L(\gamma_2(\sigma))$ and $\xi_\mu\in\partial_C L(\gamma_2(\mu))$ be such that
$$
L^0(\gamma_2(\sigma);\gamma_2(\mu-\sigma))=\langle \xi_\sigma, \gamma_2(\mu-\sigma)\rangle_{Z_2^\ast\times Z_2}=\langle\gamma_2^\ast (\xi_\sigma),\mu-\sigma\rangle_{Y^\ast\times Y},
$$
and
$$
L^0(\gamma_2(\mu);\gamma_2(\mu-\sigma))=\langle \xi_\mu, \gamma_2(\mu-\sigma)\rangle_{Z_2^\ast\times Z_2}=\langle\gamma_2^\ast( \xi_\mu),\mu-\sigma\rangle_{Y^\ast\times Y}.
$$
We apply hypotheses ${\bf (H_F^3)}$ and ${\bf (H_G^3)}$ to get
\begin{align*}
g((u,\sigma),(v,\mu))& -h((u,\sigma) ,(v,\mu))
=\beta_T(v-u) -\langle F(v)-\gamma_1^*\xi_u-F(u)+\gamma_1^*\xi_v,v-u\rangle \\& +\beta_S(\mu-\sigma)
 -\langle G(\mu)-H(\mu)\gamma_2^\ast(\xi_\mu)-G(\sigma)+H(\sigma)\gamma_2^\ast(\xi_\sigma) ,\mu-\sigma \rangle\\
\geq& 0,
\end{align*}
where $\xi_u\in \partial_C J(\gamma_1u)$ and $\xi_v\in \partial J_C(\gamma_1v)$ are such that
$$
J^0(\gamma_1(u);\gamma_1(v-u))=\langle \xi_u, \gamma_1(v-u)\rangle_{Z_1^\ast\times Z_1}=\langle\gamma_1^\ast( \xi_u),u-v\rangle_{X^\ast\times X},
$$
and
$$
J^0(\gamma_1(v);\gamma_1(v-u))=\langle \xi_v, \gamma_1(v-u)\rangle_{Z_1^\ast\times Z_1}=\langle\gamma_1^\ast( \xi_v),u-v\rangle_{X^\ast\times X}.
$$

We apply Theorem \ref{Mosco} with $a:=0$ to find an element $(u_0,\sigma_0)\in \mathcal{K}\cap D(\varphi)$ such that (\ref{BoundedSol2s}) is fulfilled.
Let $(v,\mu)\in \mathcal{K}$ be fixed. Then for any $t\in (0,1)$, we insert $(w,\lambda):=(u_0+t(v-u_0),\sigma_0)$  in \eqref{BoundedSol2s} to get
\begin{eqnarray*}
0&\leq&  t\left[ B(v,\sigma_0)-B(u_0,\sigma_0) +J^0(\gamma_1(u_0+t(v-u_0));\gamma_1(v-u_0))\right.\\
&& \left. -\langle F(u_0+t(v-u_0)),v-u_0\rangle+\phi(v)-\phi(u_0) \right].
\end{eqnarray*}
Hence, from the hemicontinuity of $F$ and upper semicontinuity of $Z_1\ni u\mapsto J^0(u;v)$, we obtain the first inequality of (\ref{eqn1}). Let $t\in (0,1)$ be arbitrary and put $(w,\lambda):=(u_0,\sigma_0+t(\mu-\sigma_0))$  in \eqref{BoundedSol2s}. We have
\begin{eqnarray*}
0&\leq& t\left[ B(u_0,\mu)-B(u_0,\sigma_0) -\langle G(\sigma_0+t(\mu-\sigma_0)),\mu-\sigma_0\rangle \right.\\
&& \left. +H(\sigma_0+t(\mu-\sigma_0))L^0(\gamma_2(\sigma_0+t(\mu-\sigma_0));\gamma_2(\mu-\sigma_0)) \right].
\end{eqnarray*}
Invoking the hemicontinuity of $G$ and upper semicontinuity of $Y\ni \sigma\mapsto H(\sigma)L^0(\gamma_2\sigma;\mu)$, we get the second inequality of  (\ref{eqn1}).

\item {\sc Case 5.} {\it Let ${\bf (H_F^1)}$ and ${\bf (H_G^2)}$ be fulfilled.}

Let us introduce the functions $g,h\colon \mathcal{K}\times\mathcal{K}\to \mathbb{R}$ given by
\begin{align*}
g((u,\sigma) ,(v,\mu)):= & B(v,\mu)-B(u,\sigma)+J^0(\gamma_1(u);\gamma_1(v)-\gamma_1(u))- \langle F(u),v-u\rangle\\
&+H(\sigma)L^0(\gamma_2(\sigma);\gamma_2(\mu)-\gamma_2(\sigma))- \langle G(\mu),\mu-\sigma \rangle
\end{align*}
and
\begin{align*}
h((u,\sigma),(v,\mu)):=& B(v,\mu)-B(u,\sigma)+J^0(\gamma_1(u);\gamma_1(v)-\gamma_1(u))-\langle F(u),v-u\rangle \\
 & + H(\sigma)L^0(\gamma_2(\sigma);\gamma_2(\mu)-\gamma_2(\sigma))-\langle G(\sigma),\mu-\sigma\rangle-\beta_G(\mu-\sigma)
\end{align*}
for all $(u,\sigma) ,(v,\mu)\in\mathcal K$.

Then, we have that  $h((u,\sigma),(u,\sigma))=-\beta_G(0)\geq 0$ for all $(u,\sigma)\in\mathcal{K}$,  the  mapping $\mathcal K\ni (u,\sigma)\mapsto g((u,\sigma),(v,\mu))\in \mathbb R$ is weakly upper semicontinuous, while the mapping $\mathcal K\ni (v,\mu)\mapsto h((u,\sigma),(v,\mu))\in\mathbb R$ is convex.
It follows from hypothesis ${\bf (H_G^2)}$ that
\begin{align*}
g((u,\sigma),(v,\mu)) -h((u,\sigma) ,(v,\mu))
=\beta_G(\mu-\sigma)  -\langle G(\mu)-G(\sigma),\mu-\sigma \rangle
\geq 0.
\end{align*}

We apply Theorem \ref{Mosco} with $a:=0$ to find an element $(u_0,\sigma_0)\in \mathcal{K}\cap D(\varphi)$ such that (\ref{BoundedSol2s}) is fulfilled.
Let $(v,\mu)\in \mathcal{K}$ be fixed. Then for any $t\in (0,1)$, we insert $(w,\lambda):=(v,\sigma_0)$ and  $(w,\lambda):=(u_0,\sigma_0+t(\mu-\sigma_0))$ in \eqref{BoundedSol2s}, respectively, and pass to the upper limit as $t\to 0^+$ for the second resulting inequality  to get the  (\ref{eqn1}).

\item {\sc Case 6.} {\it Let ${\bf (H_F^1)}$ and ${\bf (H_G^3)}$ be fulfilled.}

Let us introduce the functions $g,h\colon \mathcal{K}\times\mathcal{K}\to \mathbb{R}$ given by
\begin{align*}
g((u,\sigma),(v,\mu)):=& B(v,\mu)-B(u,\sigma)+J^0(\gamma_1(u);\gamma_1(v-u))-\langle F(u),v-u\rangle \\
 & + H(\mu)L^0(\gamma_2(\mu);\gamma_2(\mu-\sigma))-\langle G(\mu),\mu-\sigma\rangle
\end{align*}
and
\begin{align*}
h((u&,\sigma),(v,\mu)):= B(v,\mu)-B(u,\sigma)+J^0(\gamma_1(u);\gamma_1(v-u))-\langle F(u),v-u\rangle \\
 & + H(\sigma)L^0(\gamma_2(\sigma);\gamma_2(\mu-\sigma))-\langle G(\sigma),\mu-\sigma\rangle -\beta_S(\mu-\sigma).
\end{align*}
for all $(u,\sigma) ,(v,\mu)\in\mathcal K$.

Then, we have that  $h((u,\sigma),(u,\sigma))=-\beta_S(0)\geq 0$ for all $(u,\sigma)\in\mathcal{K}$,  the  mapping $\mathcal K\ni (u,\sigma)\mapsto g((u,\sigma),(v,\mu))\in \mathbb R$ is weakly upper semicontinuous, while the mapping $\mathcal K\ni (v,\mu)\mapsto h((u,\sigma),(v,\mu))\in\mathbb R$ is convex.
It follows from hypothesis ${\bf (H_G^3)}$ that
$g((u,\sigma),(v,\mu)) -h((u,\sigma) ,(v,\mu))
\geq 0.$

We apply Theorem \ref{Mosco} with $a:=0$ to find an element $(u_0,\sigma_0)\in \mathcal{K}\cap D(\varphi)$ such that (\ref{BoundedSol2s}) is fulfilled.
Let $(v,\mu)\in \mathcal{K}$ be fixed. Then for any $t\in (0,1)$, we insert $(w,\lambda):=(v,\sigma_0)$ and  $(w,\lambda):=(u_0,\sigma_0+t(\mu-\sigma_0))$   in \eqref{BoundedSol2s}, respectively,   and pass to the upper limit as $t\to 0^+$ for the second resulting inequality  to get the  (\ref{eqn1}).

\item {\sc Case 7.} {\it Let ${\bf (H_F^2)}$ and ${\bf (H_G^1)}$ be fulfilled.}

Let us introduce the functions $g,h\colon \mathcal{K}\times\mathcal{K}\to \mathbb{R}$ given by
\begin{align*}
g((u,\sigma),(v,\mu)):=& B(v,\mu)-B(u,\sigma)+J^0(\gamma_1(u);\gamma_1(v-u))-\langle F(v),v-u\rangle \\
 & + H(\sigma)L^0(\gamma_2(\sigma);\gamma_2(\mu-\sigma))-\langle G(\sigma),\mu-\sigma\rangle
\end{align*}
and
\begin{align*}
h((u&,\sigma),(v,\mu)):= B(v,\mu)-B(u,\sigma)+J^0(\gamma_1(u);\gamma_1(v-u))-\langle F(u),v-u\rangle \\
 &-\beta_F(v-u) + H(\sigma)L^0(\gamma_2(\sigma);\gamma_2(\mu-\sigma))-\langle G(\sigma),\mu-\sigma\rangle.
\end{align*}
for all $(u,\sigma) ,(v,\mu)\in\mathcal K$.

Then, we have that  $h((u,\sigma),(u,\sigma))=-\beta_F(0)\geq 0$ for all $(u,\sigma)\in\mathcal{K}$,  the  mapping $\mathcal K\ni (u,\sigma)\mapsto g((u,\sigma),(v,\mu))\in \mathbb R$ is weakly upper semicontinuous, while the mapping $\mathcal K\ni (v,\mu)\mapsto h((u,\sigma),(v,\mu))\in\mathbb R$ is convex.
It follows from hypothesis ${\bf (H_G^3)}$ that
$g((u,\sigma),(v,\mu)) -h((u,\sigma) ,(v,\mu))
\geq 0.$

We apply Theorem \ref{Mosco} with $a:=0$ to find an element $(u_0,\sigma_0)\in \mathcal{K}\cap D(\varphi)$ such that (\ref{BoundedSol2s}) is fulfilled.
Let $(v,\mu)\in \mathcal{K}$ be fixed. Then for any $t\in (0,1)$, we insert $(w,\lambda):=(tv+(1-t)u_0,\sigma_0)$ and  $(w,\lambda):=(u_0,\mu)$   in \eqref{BoundedSol2s}, respectively,   and pass to the upper limit as $t\to 0^+$ for the first resulting inequality  to get the  (\ref{eqn1}).

\item {\sc Case 8.} {\it Let ${\bf (H_F^3)}$ and ${\bf (H_G^1)}$ be fulfilled.}

Consider the functions $g,h\colon \mathcal{K}\times\mathcal{K}\to \mathbb{R}$ given by
\begin{align*}
g((u,\sigma),(v,\mu)):=& B(v,\mu)-B(u,\sigma)+J^0(\gamma_1(v);\gamma_1(v-u))-\langle F(v),v-u\rangle \\
 & + H(\sigma)L^0(\gamma_2(\sigma);\gamma_2(\mu-\sigma))-\langle G(\sigma),\mu-\sigma\rangle
\end{align*}
and
\begin{align*}
h((u,\sigma),(v,\mu)):=& B(v,\mu)-B(u,\sigma)+J^0(\gamma_1(u);\gamma_1(v-u))-\langle F(u),v-u\rangle \\
 &-\beta_T(v-u) + H(\sigma)L^0(\gamma_2(\sigma);\gamma_2(\mu-\sigma))-\langle G(\sigma),\mu-\sigma\rangle
\end{align*}
for all $(u,\sigma) ,(v,\mu)\in\mathcal K$.

Then, we have that  $h((u,\sigma),(u,\sigma))=-\beta_T(0)\geq 0$ for all $(u,\sigma)\in\mathcal{K}$,  the  mapping $\mathcal K\ni (u,\sigma)\mapsto g((u,\sigma),(v,\mu))\in \mathbb R$ is weakly upper semicontinuous, while the mapping $\mathcal K\ni (v,\mu)\mapsto h((u,\sigma),(v,\mu))\in\mathbb R$ is convex.
It follows from hypothesis ${\bf (H_G^3)}$ that
$g((u,\sigma),(v,\mu)) -h((u,\sigma) ,(v,\mu))
\geq 0.$

We apply Theorem \ref{Mosco} with $a:=0$ to find an element $(u_0,\sigma_0)\in \mathcal{K}\cap D(\varphi)$ such that (\ref{BoundedSol2s}) is fulfilled.
Let $(v,\mu)\in \mathcal{K}$ be fixed. Then for any $t\in (0,1)$, we insert $(w,\lambda):=(tv+(1-t)u_0,\sigma_0)$ and  $(w,\lambda):=(u_0,\mu)$   in \eqref{BoundedSol2s}, respectively,   and pass to the upper limit as $t\to 0^+$ for the first resulting inequality  to get the  (\ref{eqn1}).
\end{description}
This completes the proof of the lemma.\qed
\end{proof}

We end this section with an existence result when at least one of the constraint sets  is unbounded. To this end, we need a coercivity condition. Here, and hereafter, we consider the space $X\times Y$ is endowed with the norm $\|(u,\sigma)\|:=\sqrt{\|u\|^2_X+\|\sigma\|^2_Y}$.  We consider the following coercivity condition:

\noindent ${\bf (C)}$\quad There exist $\bar{w}_0\in K\cap D(\phi)$ and $\bar{\tau}_0\in\Lambda$ such that  $${\displaystyle \frac{2B(u,\sigma)-B(\bar{w}_0,\sigma)-B(u,\bar{\tau}_0)+\langle F(u),\bar{w}_0-u\rangle+\langle G(\sigma),\bar{\tau}_0-\sigma \rangle}{\|(u,\sigma)\|}\to \infty }, \mbox{ as }\|(u,\sigma)\|\to\infty.$$

We also consider ${\bf (H_0')}$ to be set of assumptions obtained from  ${\bf (H_0)}$ by replacing (ii) with the slightly stronger condition:

 (ii') $J\colon Z_1\to\mathbb{R}$ and $L:Z_2\to \mathbb{R}$ are locally Lipschitz functionals and there exist $\alpha_J>0$ and  $\alpha_L>0$ such that
$$
\sup_{x\in Z_1} c_x\leq \alpha_J\mbox{ and }\sup_{y\in Z_2}c_y\leq \alpha_L,
$$
where $c_x>0$ (resp. $c_y>0$)  denotes the Lipschitz constant of $J$ (resp. $L$) near the point $x\in Z_1$ (resp. $y\in Z_2$).

\begin{theorem}\label{C-Z_MainRes1}
Let $X,Y$ be real reflexive Banach spaces. Assume that $K\subseteq X$,  $\Lambda\subseteq Y$ are nonempty closed convex sets such that either $K\times \Lambda$  is bounded or ${\bf (C)}$ is satisfied.  If, in addition, ${\bf (H_0')}$, ${\bf (H_B)}$, ${\bf (H_F^j)}$, ${\bf (H_G^k)}$ with $j,k\in\{1,2,3\}$, then problem (\ref{eqn1}) possesses at least one solution in $K\times \Lambda$.
\end{theorem}

\begin{proof}
If $K\times \Lambda$ is bounded, then the desired conclusion can be obtained directly by using Lemmas~\ref{BoundedSol1} and~\ref{BoundedSol2}.

So, we assume now that  $K\times\Lambda$ is unbounded. Let $r_0>0$ be sufficiently large such that $K_r:=K\cap \bar{B}_X(\bar{w}_0,r)$ and $\Lambda_r:=\Lambda\cap \bar{B}_Y(\bar{\tau}_0,r)$ are both nonempty for all $r\ge r_0$.  Let $r\ge r_0$ be arbitrary fixed. Next, we consider the following intermediate problem:

{\it Find $(u,\sigma)\in \left( K_r\cap D(\phi) \right)\times\Lambda_r$ such that
\begin{align}\label{eqn3.7} \left\{
\begin{array}{ll}
B(v,\sigma)-B(u,\sigma)+\phi(v)-\phi(u)+J^0(\gamma_1 u;\gamma_1 (v- u))\geq \langle F(u),v-u\rangle,\\
B(u,\mu)-B(u,\sigma)+H(\sigma)L^0(\gamma_2 \sigma;\gamma_2( \mu-\ \sigma))\geq \langle G(\sigma),\mu-\sigma \rangle,
\end{array}
\right.
\end{align}
for all $v\in K_r$ and all $\mu\in \Lambda_r$.}

\noindent
In virtue of Lemmas \ref{BoundedSol1} and \ref{BoundedSol2}, we can see that problem (\ref{eqn3.7}) has at least one solution, say  $(u_r,\sigma_r)\in K_r\times \Lambda_r$.

We claim that there exists a constant $r_1\ge r_0$ such
that each solution of problem (\ref{eqn3.7}) with $r=r_1$ satisfies the following inequality
\begin{eqnarray}\label{innine}
\max\{\|u_{r_1}-\bar{w}_0\|_X,\|\sigma_{r_1}-\bar{\tau}_0\|_Y\}<r_1.
\end{eqnarray}
Arguing by contradiction, we assume that for each $r\ge r_0$ it holds \begin{eqnarray*}
\max\{\|u_r-\bar{w}_0\|_X,\|\sigma_r-\bar{\tau_0}\|_Y\}=r.
\end{eqnarray*}
Then, plugging $(v,\mu):=(\bar{w}_0,\bar{\tau}_0)$ into (\ref{eqn3.7}) we find that
$$
\left\{
\begin{array}{l}
B(\bar{w}_0,\sigma_r)-B(u_r,\sigma_r)+\phi(\bar{w}_0)-\phi(u_r)+J^0(\gamma_1(u_r);\gamma_1(\bar{w}_0-u_r))\geq
\langle F(u_r),\bar{w}_0-u_r\rangle,
\\
B(u_r,\bar{\tau}_0)-B(u_r,\sigma_r)+H(\sigma_r)L^0(\gamma_2(\sigma_r);\gamma_2(\bar{\tau}_0-\sigma_r))\geq
\langle G(\sigma_r),\bar{\tau}-\sigma_r\rangle.
\end{array}
\right.
$$
Summing up the last two inequalities and  we get
\begin{align}\label{Coercive1}
 \phi(\bar{w}_0)&-\phi(u_r)-\langle \zeta_{u_r}, \gamma_1(\bar{w}_0-u_r)\rangle-H(\sigma_r)\langle \zeta_{\sigma_r}\gamma_2(\bar{\tau}_0-\sigma_r)\rangle\nonumber
 \geq    2B(u_r,\sigma_r)-B(u_r,\bar{\tau}_0)\\
 &-B(\bar{w}_0,\sigma_r)+\langle F(u_r),\bar{w}_0-u_r \rangle -\langle G(\sigma_r),\bar{\tau}_0-\sigma_r \rangle,
\end{align}
where $\zeta_{u_r}\in \partial_CJ(\gamma_1 (u_r))$ and $\zeta_{\sigma_r}\in \partial_CL(\gamma_2(\sigma_r))$ are such that
\begin{align*}
&\langle \zeta_{u_r},\gamma_1(\bar{w_0}-u_r)\rangle_{Z_1^*\times Z_1}=\max_{\zeta\in\partial_C J(\gamma_1(u_r))}\langle \zeta,\gamma_1(\bar{w}_0-u_r)\rangle_{Z_1^*\times Z_1},\\
&\langle \zeta_{\sigma_r},\gamma_2(\bar{\tau}_0-\sigma_r)\rangle_{Z_2^*\times Z_2}=\max_{\zeta\in\partial_C L(\gamma_2(\sigma_r))}\langle \zeta,\gamma_2(\bar{\tau}_0-\sigma_r)\rangle_{Z_2^*\times Z_2}.
\end{align*}
Since any convex and l.s.c. functional is bounded below (see, e.g, \cite[Proposition 1.10]{Brezis2011}) there exist constants $\alpha_\phi,\beta_\phi\ge0$ such that
\begin{eqnarray}\label{eqn3.9}
\phi(u)\ge -\alpha_\phi\|u\|_X-\beta_\phi, \ \forall u\in X.
\end{eqnarray}
Inserting (\ref{eqn3.9}) into (\ref{Coercive1}), we have
\begin{align}\label{eqn3.10}
 2B(u_r,\sigma_r)& -B(u_r,\bar{\tau}_0)-B(\bar{w}_0,\sigma_r) \langle F(u_r),\bar{w}_0-u_r\rangle-\langle G(\sigma_r),\bar{\tau}_0-\sigma_r \rangle\nonumber
 \leq \phi(\bar{w}_0)+\alpha_\phi\|u\|_X\\
 &+\beta_\phi +|\langle \zeta_{u_r}, \gamma_1(\bar{w}_0-u_r)\rangle|+H(\sigma_r)|\langle \zeta_{\sigma_r},\gamma_2(\bar{\tau}_0-\sigma_r)\rangle|.
\end{align}
Dividing \eqref{eqn3.10} by $\sqrt{\|u_r\|^2_X+\|\sigma_r\|^2_Y}$ and letting $r\to \infty$, it gives
 \begin{align*}
 \infty=&\lim_{r\to \infty}\frac{2B(u_r,\sigma_r)-B(u_r,\bar{\tau}_0)-B(\bar{w}_0,\sigma_r) +\langle F(u_r),\bar{w}_0-u_r\rangle+\langle G(\sigma_r),\bar{\tau}_0-\sigma_r \rangle}{\|(u_r,\sigma_r)\|}\nonumber\\
 \leq &  \lim_{r\to\infty}\frac{\phi(\bar{w}_0)+\alpha_\phi\|u\|_X+\beta_\phi +|\langle \zeta_{u_r}, \gamma_1(\bar{w}_0-u_r)\rangle|+H(\sigma_r)|\langle \zeta_{\sigma_r},\gamma_2(\bar{\tau}_0-\sigma_r)\rangle|}{\|(u_r,\sigma_r)\|}\nonumber\\
 \le &\alpha_\phi+\alpha_J\|\gamma_1\|+\alpha_Lm_H\|\gamma_2\|,
\end{align*}
which is obviously a contradiction. Therefore, we conclude that there exists a constant $r_1>0$ such that inequality (\ref{innine}) is satisfied.

Let $(u_{r_1},\sigma_{r_1})\in K_{r_1}\times \Lambda_{r_1}$ be a solution to problem (\ref{eqn3.7}) with $r=r_1$. Furthermore, we are going to show that $(u_{r_1},\sigma_{r_1})$ solves problem (\ref{eqn1}) as well.  Let $(v,\mu)\in K\times \Lambda$ be arbitrarily fixed. It follows from (\ref{innine}) (i.e., $(u_{r_1},\sigma_{r_1})\in B_X(\bar{w}_0,r_1)\times B_Y(\bar{\tau}_0,r_1)$) that there exists a constant $t\in (0,1)$ such that
$$
(w,\lambda):=(u_{r_1},\sigma_{r_1})+t(v-u_{r_1},\mu-\sigma_{r_1})\in K_{r_1}\times \Lambda_{r_1}.
$$
Inserting $(u,\sigma)=(u_{r_1},\sigma_{r_1})$ and $(v,\mu)=(w,\lambda)$ into the first inequality of (\ref{eqn3.7}), we have
\begin{align*}
t\langle F(u_{r_1},v-u_{r_1}) \rangle
 \leq&  B(w,\sigma_{r_1})-B((u_{r_1},\sigma_{r_1}))+\phi(w)-\phi(u_{r_1}) +J^0(\gamma_1(u_{r_1});\gamma_1(w-u_{r_1}))\\
 \leq& t\left[ B(v,\sigma_{r_1})-B(u_{r_1},\sigma_{r_1})+\phi(v)-\phi(u_{r_1}) +J^0(\gamma_1(u_r);\gamma_1(v-u_r)) \right].
\end{align*}
Dividing by $t>0$ we infer directly that $(u_{r_1},\sigma_{r_1})$ satisfies the first inequality of (\ref{eqn1}). Likewise, we also can prove that $(u_{r_1},\sigma_{r_1})$ satisfies the second inequality of (\ref{eqn1}). This means that $(u_{r_1},\sigma_{r_1})\in K_{r_1}\times \Lambda_{r_1}$ is  a solution to problem (\ref{eqn1}), hence the proof is now complete.
\qed

\end{proof}

\section{Applications to Contact Mechanics}

Throughout this section ${\cal S}^m$  denotes the linear space of second order symmetric tensors on $\RR^m$, i.e., ${\cal S}^m=\RR^{m\times m}_s$. The inner products and the corresponding norms on $\mathbb{R}^m$ and ${\cal S}^m$ are defined by
$$
\bu\cdot \bv = u_iv_i, \quad |\bv|=(\bv\cdot \bv)^{1/2}, \quad \bu:=(u_i),\ \bv:=(v_i)\in\RR^m,
$$
and, respectively,
$$
\btau \cdot \bsigma =\tau_{ij}\sigma_{ij}, \quad |\btau|=(\btau\cdot\btau)^{1/2}, \quad \btau:=(\tau_{ij}), \ \bsigma:=(\sigma_{ij})\in {\cal S}^m.
$$
Here and hereafter, $m$ is a positive integer which stands for the dimension of the spatial variable, indices $i$ and $j$ run from $1$ to $m$ and the summation convention of the repeated indices is adopted. For a bounded open set $\Omega\subset \RR^m$ with sufficiently smooth boundary $\Gamma$ (normally, we assume that $\Gamma$ is Lipschitz continuous) we denote by $\bnu$ the outward unit vector on $\Gamma$ and we introduce the following  function spaces which will play a key role in applications to nonsmooth mechanics problems
$$
H:=L^2\left(\Omega;\RR^m\right),\quad {\cal H}:=\left\{\btau = (\tau_{ij}): \ \ \tau_{ij}=\tau_{ji}\in L^2\left(\Omega\right)\right\}=L^2\left(\Omega;{\cal S}^m\right),
$$
$$
H_1:=\left\{\bu\in H: \ \ \bvarepsilon(\bu)\in {\cal H}\right\}=H^1\left(\Omega;\RR^m\right),\quad {\cal H}_1:=\left\{\btau\in {\cal H}: \ \ {\rm Div\;}\btau\in H\right\},
$$
where $\bvarepsilon$ and ${\rm Div}$ are the deformation operator and the divergence operator, respectively, and are defined in the following way
$$
\bvarepsilon (\bu):=(\varepsilon_{ij}(\bu)),\quad \varepsilon_{ij}(\bu):=\frac{1}{2}(u_{i,j}+u_{j,i}), \qquad {\rm Div\;} \btau=(\tau_{ij,j}),
$$
accordingly.  The index following a comma represents the partial derivative with respect to the corresponding component of $x\in \Omega$, i.e., $u_{i,j}=\de u_i/\de x_j$. Keep in mind that the above mentioned spaces are Hilbert spaces endowed with the following corresponding inner products
$$
(\bu,\bv)_H:=\int_\Omega \bu\cdot \bv\; dx,\quad (\btau,\bsigma)_{\cal H}:=\int_\Omega \btau\cdot\bsigma \; dx, \quad (\bu,\bv)_{H_1}:=(\bu,\bv)_H+(\bvarepsilon(\bu),\bvarepsilon(\bv))_{\cal H},
$$
$$
(\btau,\bsigma)_{{\cal H}_1}:=(\btau,\bsigma)_{\cal H}+({\rm Div\; }\btau,{\rm Div\; }\bsigma)_{H}.
$$
We recall that the trace operator $\gamma:H^1(\Omega;\RR^m)\rightarrow H^{1/2}(\Gamma;\RR^m)\subset L^{2}(\Gamma;\RR^m)$ is linear and compact. Sometimes, for simplicity, we will omit to write $\gamma$ to indicate the Sobolev trace on the boundary, i.e., writing $\bv$ instead of $\gamma \bv$. Also, for a given $\bv \in H^{1/2}(\Gamma;\RR^m)$ we denote by $v_{\nu}$ and $\bv_{\tau}$ the normal and the tangential components of $\bv$ on the boundary, i.e., $v_{\nu}:=\bv\cdot \bnu$ and $\bv_{\tau}:=\bv-v_{\nu}\bnu$, respectively. Similarly, for a tensor field  $\bsigma$, we define $\sigma_{\nu}$ and $\bsigma_{\tau}$ to be the normal and the tangential components of the Cauchy vector field $\bsigma\bnu$, that is $\sigma_\nu:=(\bsigma\bnu)\cdot \bnu$ and $\bsigma_\tau:=\bsigma\bnu-\sigma_\nu\bnu$, respectively. Recall that the following Green formula holds
\begin{equation}\label{GreenEqn1}
(\bsigma,\bvarepsilon(\bv))_{\cal H}+({\rm Div\; }\bsigma,\bv)_H=\int_{\Gamma} (\bsigma\bnu)\cdot \bv \;d\Gamma,  \mbox{ for all }\bv\in H_1.
\end{equation}

We consider the following contact model involving a deformable body which occupies a bounded domain $\Omega\subset\mathbb{R}^m$ ($m\ge 2$) with Lipschitz boundary $\Gamma$, which is partitioned into four disjoint measurable parts $\Gamma_1$, $\Gamma_2$ and $\Gamma_3$.

\noindent ${\bf (P)}$ Find a displacement $\bu:\Omega\rightarrow \RR^m$ and a stress tensor $\bsigma:\Omega\rightarrow {\cal S}^m$ such that
\begin{align}
&-{\rm Div\;} \bsigma =\mathbf{f}_0& \mbox{ in }& \Omega,\label{P1Eqn1}\\
&\bsigma \in \de \chi(\bvarepsilon(\bu))&  \mbox{ in }&\Omega,\label{P1Eqn2}\\
&\bu=0& \mbox{ on }& \Gamma_1,\label{P1Eqn3}\\
&\bsigma\bnu=\mathbf{f}_2& \mbox{ on }& \Gamma_2,\label{P1Eqn4}\\
&u_\nu=0, -\bsigma_\tau\in \partial_C j(\bx,\bu_\tau)& \mbox{ on }& \Gamma_3^a,\label{P1Eqn5}\\
&\bsigma_\tau=0, \sigma_\nu\leq 0, u_\nu\leq 0, \sigma_\nu u_\nu=0&  \mbox{ on }&\Gamma_3^b.\label{P1Eqn6}
\end{align}

Problem ${\bf (P)}$ describes the contact between a deformable body and a foundation. Relation \eqref{P1Eqn1} represents the equilibrium equation, showing that the body is subjected to volume forces of density $\mathbf{f}_0$. The behaviour of the material is described by a nonlinear constitutive law expressed as a
(convex) subdifferential inclusion, namely, \eqref{P1Eqn2}. On $\Gamma_1$ the body is clamped, therefore the displacement vanishes here. Equation \eqref{P1Eqn4} shows that surface tractions of density $\mathbf{f}_2$ act on $\Gamma_2$.  On $\Gamma_3$ the body may come in frictional contact with the foundation and this contact is modelled considering there are two potential  contact zones: $\Gamma_3^a$ where a nonmonotone friction law is considered; and $\Gamma_3^b$ where the contact is frictionless and satisfies Signorini unilateral contact model. Note that we allow the case of only one contact zone, i.e., $meas(\Gamma_3^a)=0$ or $meas(\Gamma_3^b)=0$.

For examples of nonlinear constitutive laws of the form $\eqref{P1Eqn2}$ and nonmonotone friction laws of the form \eqref{P1Eqn5} we refer the reader to \cite{Cos-Csi-Var15,Cos-Mat12,Nan-Pan,Pan93,Sof-Mat12}.

We assume that the constitutive function $\chi$ and the initial data $\mathbf{f}_0,\mathbf{f}_2$ satisfy the following conditions:

\begin{enumerate}[$(\mathcal{H}_1)$]

\item $\mathbf{f}_0\in L^2(\Omega;\mathbb{R}^m)$  and $\mathbf{f}_2\in L^2(\Gamma_2;\mathbb{R}^m)$;

\item $\chi:\mathcal{S}^m\to\mathbb{R}$ is a convex and l.s.c. functional with the property that there exist $\alpha,\beta\in (0,1)$ such that
$$
\alpha |\bmu|^2\leq \chi(\bmu)\leq \beta |\bmu|^2,\ \forall \bmu\in \mathcal{S}^m;
$$

\item If $meas(\Gamma_3^a)>0$, then  $j:\Gamma_3^a\times\mathbb{R}\to\mathbb{R}$ is functional enjoying the following properties:
\begin{enumerate}[{\rm (i)}]
 \item $\bx\mapsto j(\bx,t)$ is measurable on $\Gamma_3^a$ for all $t\in\mathbb{R}$;

 \item $j(\bx,0)\in L^1(\Gamma_3^a)$;

 \item there exists $p\in L^2(\Gamma_3^a)$ such that
 $$
 |j(\bx,t_1)-j(\bx,t_2)|\leq p(x)|t_1-t_2|, \mbox{ for a.e. $\bx\in\Gamma_3^a$ and all $t_1,t_2\in\mathbb{R}$}.
 $$
\end{enumerate}

\end{enumerate}

In order to derive the variational formulation of problem ${\bf (P)}$ assume $\bu$ and $\bsigma$ are regular functions that satisfy \eqref{P1Eqn1}-\eqref{P1Eqn6}. Multiplying \eqref{P1Eqn1} by $\bv-\bu$ and integrating over $\Omega$ we apply  Green's formula to find that
$$
(\mathbf{f}_0,\bv-\bu)_{H}=-(\Div \bsigma, \bv-\bu)_{H}=(\bsigma,\bvarepsilon(\bv-\bu))_{\mathcal{H}}-\int_{\Gamma}(\bsigma\bnu)\cdot(\bv-\bu) d\Gamma,
$$
for all $\bv\in X:=\{\bw\in H_1:\ \bw=0 \mbox{ a.e. on }\Gamma_1 \}$. It is well known that $X$ is a closed subspace of $H_1$, therefore it is a Hilbert space endowed with the inner product
$$
(\bu,\bv)_X:=(\bvarepsilon(\bu),\bvarepsilon(\bv))_{\mathcal{H}}.
$$
Keeping in mind the boundary conditions on $\Gamma_3$ we define the set of {\it admissible displacements}
$$
\mathcal{K}_0:=\{\bw\in X:\ w_\nu=0 \mbox{ on }\Gamma_3^a \mbox{ and }w_\nu\leq 0 \mbox{ on }\Gamma_3^b\},
$$
which is nonempty, unbounded, closed and convex subset of $X$. Keeping in mind that $$(\bsigma\bnu)\cdot(\bv-\bu)=\sigma_\nu(v_\nu-u_\nu)+\bsigma_\tau\cdot(\bv_\tau-\bu_\tau)$$
for all $\bv\in \mathcal{K}_0$ we have
$$
-\int_{\Gamma_1}\bsigma\bnu\cdot(\bv-\bu)\; d\Gamma=0,
$$
$$
-\int_{\Gamma_2}\bsigma\bnu\cdot(\bv-\bu)\; d\Gamma=-\int_{\Gamma_2} \mathbf{f}_2\cdot(\bv-\bu)\; d\Gamma,
$$
$$
-\int_{\Gamma_3^a}\bsigma\bnu\cdot(\bv-\bu)\; d\Gamma=\int_{\Gamma_3^a}-\bsigma_\tau\cdot(\bv_\tau-\bu_\tau)\; d\Gamma\leq\int_{\Gamma_3^a} j^0_{,2}(\bx,\bu_\tau;\bv_\tau-\bu_\tau)\; d\Gamma,
$$
and
$$
-\int_{\Gamma_3^b}\bsigma\bnu\cdot(\bv-\bu)\; d\Gamma=-\int_{\Gamma_3^b} \sigma_\nu(v_\nu-u_\nu)\;d\Gamma=-\int_{\Gamma_3^b} \sigma_\nu v_\nu\;d\Gamma\leq 0.
$$
Thus,
\begin{equation}\label{VarFor1}
(\bsigma,\bvarepsilon(\bv-\bu))_{\mathcal{H}}+\int_{\Gamma_3^a}j^0_{,2}(\bx,\bu_\tau;\bv_\tau-\bu_\tau)\;d\Gamma\geq (\mathbf{f}_0,\bv-\bu)_{H}+\int_{\Gamma_2}\mathbf{f}_2\cdot(\bv-\bu)\; d\Gamma,\quad\forall \bv\in \mathcal{K}_0.
\end{equation}
Since the stress tensor $\bsigma$ satisfies \eqref{VarFor1} it is natural to define now the set of {\it admissible stress tensors with respect to a displacement $\bw\in \mathcal{K}_0$} to be the following subset of $\mathcal{H}$
$$
\Theta(\bw):=\left\{\bmu\in\mathcal{H}:\ (\bmu,\bvarepsilon(\bv))_{\mathcal{H}}+\int_{\Gamma_3^a}j^0_{,2}(\bx,\bw_\tau;\bv_\tau)\;d\Gamma\geq (\mathbf{f}_0,\bv)_H+\int_{\Gamma_2}\mathbf{f}_2\cdot\bv\; d\Gamma, \ \forall \bv\in \mathcal{K}_0\right\}.
$$
We define the separable bipotential $b:\mathcal{S}^m\times\mathcal{S}^m\to (-\infty,+\infty]$ by
$$
b(\btau,\bmu):=\chi(\btau)+\chi^\ast(\bmu).
$$
We point out the fact that $b$ connects the constitutive law \eqref{P1Eqn2}, the function $\chi$ and its Fenchel conjugate $\chi^\ast$ as
$$
b(\bvarepsilon(\bu),\bsigma)=\bsigma\cdot \bvarepsilon(\bu)
$$
and
$$
b(\bvarepsilon(\bv),\bmu)\geq\bmu\cdot \bvarepsilon(\bu), \quad \forall \bv\in X,\forall \bmu\in \mathcal{H}.
$$
Note that, if $(\mathcal{H}_3)$ holds, then according to \cite[Lemma 1]{Mat13} one has
$$
(1-\beta)|\bmu|^2\leq \chi^\ast(\bmu)\leq \frac{1}{4\alpha}|\bmu|^2,\ \forall \bmu\in\mathcal{S}^m.
$$
Therefore, we have
$$
\chi(\bmu(\cdot))\in L^1(\Omega) \mbox{ and }\chi^\ast(\bmu(\cdot))\in L^1(\Omega), \ \forall \bmu\in \mathcal{H}.
$$
Thus, the coupling function $B:X\times \mathcal{H}\to \mathbb{R}$ defined  via the bipotential $b$
$$
B(\bv,\bmu):=\int_\Omega b(\bvarepsilon(\bv),\bmu)\; dx,
$$
is well defined and the following estimates hold
\begin{equation}\label{BipVarForm2}
B(\bu,\bsigma)=(\bsigma,\bvarepsilon(\bu))_\mathcal{H} \mbox{ and }B(\bv,\bmu)\geq (\bmu,\bvarepsilon(\bv))_\mathcal{H}, \forall \bv\in X,\forall \bmu\in \mathcal{H}.
\end{equation}
Moreover, there exists a positive constant $C=C(\alpha,\beta)$ such that
\begin{equation}
B(\bv,\bmu)\geq C(\|\bv\|^2_X+\|\bmu\|_{\mathcal{H}}^2),\ \forall \bv\in X,\forall \bmu\in\mathcal{H}.
\end{equation}
Choosing $\bv=\bu+\bw$ in \eqref{VarFor1} we deduce that $\bsigma\in \Theta(\bu)$, therefore $\Theta(\bu)\neq\varnothing$. Keeping in mind the definition of $\Theta(\bu)$ and \eqref{BipVarForm2} we get
\begin{equation}\label{BipVarForm3}
B(\bu,\bmu)+\int_{\Gamma_3^a} j_{,2}^0(\bx,\bu_\tau;\bu_\tau)\; d\Gamma\geq (\mathbf{f}_0,\bu)_H+\int_{\Gamma_2}\mathbf{f}_2\cdot\bu\; d\Gamma.
\end{equation}
On the other hand, taking $\bv:=\mathbf{0}_X$ in \eqref{VarFor1}
\begin{equation}\label{BipVarForm4}
-B(\bu,\bsigma)+\int_{\Gamma_3^a} j_{,2}^0(\bx,\bu_\tau;-\bu_\tau)\; d\Gamma\geq- (\mathbf{f}_0,\bu)_H-\int_{\Gamma_2}\mathbf{f}_2\cdot\bu\; d\Gamma.
\end{equation}
Adding \eqref{BipVarForm3} and \eqref{BipVarForm4} we get
\begin{equation}\label{BipVarForm5}
B(\bu,\bmu)-B(\bu,\bsigma)\geq -\int_{\Gamma_3^a} [ j^0_{,2}(\bx,\bu_\tau;\bu_\tau)+ j^0_{,2}(\bx,\bu_\tau;-\bu_\tau)]\; d\Gamma,\ \forall \bmu\in\Theta(\bu).
\end{equation}
Using relations \eqref{VarFor1},  \eqref{BipVarForm2} and \eqref{BipVarForm5} we derive the following {\it variational formulation in terms of bipotentials} of problem ${\bf (P)}$:

 ${\bf (P)}$: Find $\bu\in \mathcal{K}_0$ and $\bsigma\in \Theta(\bu)$ such that
\begin{equation}\label{BipVarForm}
\left\{
\begin{array}{l}
{\displaystyle B(\bv,\bsigma)-B(\bu,\bsigma)+\int_{\Gamma_3^a} j_{,2}^0(\bx,\bu_\tau;\bv_\tau-\bu_\tau)\; d\Gamma\geq (\mathbf{f}_0,\bv-\bu)_H+\int_{\Gamma_2}\mathbf{f}_2\cdot(\bv-\bu)\; d\Gamma}\\
\medskip
B(\bu,\bmu)-B(\bu,\bsigma)\geq 0,
\end{array}
\right.
\end{equation}
for all $(\bv,\bmu)\in \mathcal{K}_0\times \Theta(\bu)$.

Note that, if $(\bu,\bsigma)$ is a {\it weak solution} of problem ${\bf (P)}$, i.e., it solves \eqref{BipVarForm}, then \eqref{VarFor1} and \eqref{BipVarForm5} are automatically fulfilled as
$$
B(\bv,\bsigma)-B(\bu,\bsigma)\geq (\bsigma,\bvarepsilon(\bv-\bu))_{\mathcal{H}},\ \forall \bv\in \mathcal{K}_0,
$$
 and
$$
0=j_{,2}^0(\bx,\bu_\tau;\mathbf{0}_X)=j_{,2}^0(\bx,\bu_\tau;\bu_\tau-\bu_\tau)\leq j_{,2}^0(\bx,\bu_\tau;\bu_\tau)+j_{,2}^0(\bx,\bu_\tau;-\bu_\tau).
$$

\begin{theorem}\label{ExistenceWeakSol}
Assume $(\mathcal{H}_1), (\mathcal{H}_2)$ and $(\mathcal{H}_3)$ hold. Then problem ${\bf (P)}$ possesses at least one weak solution.
\end{theorem}

\begin{proof}
Let us define the function spaces $Y:=\mathcal{H}$, $Z_1:=L^2(\Gamma_3^a;\mathbb{R}^m)$, $Z_2:=L^2(\Omega;\mathcal{S}^m)$. Moreover, assume that  $\gamma_2:Y\to Z_2$ is the embedding operator and $\gamma_1:X\to Z_1$  is defined by
$$
\gamma_1(\bu):=(\gamma(\bu_\tau))|_{\Gamma_3^a},
$$
where $\gamma:X\to H^{1/2}(\Gamma;\mathbb{R}^m)$ is the trace operator. Also define $J: Z_1\to\mathbb{R}$,  $\phi:X\to(-\infty,\infty]$, $H:Y\to (0,\infty)$, $L:Z_2\to \mathbb{R}$, $F:X\to X^\ast$ and $G:Y\to Y^\ast$ by the following instructions
$$
J(\by):=\int_{\Gamma_3^a} j(\bx,\by(\bx)) d\Gamma,\  \phi:=I_{\mathcal{K}_0},\  H\equiv 1,\  L\equiv 0, \ G\equiv 0_{Y^\ast}
$$
and
$$
\langle F(\bu),\bv \rangle:=(\mathbf{f}_0,\bv)_{\mathcal{H}}+\int_{\Gamma_2} \mathbf{f}_2 \cdot \bv d\Gamma,
$$
respectively.

By virtue of definition of $B\colon X\times Y \to\mathbb{R}$, it is easy to check that ${\bf (H_B)}$, ${\bf (H_F^1)}$ and ${\bf (H_G^1)}$ are fulfilled. Furthermore, the Aubin-Clarke Theorem (see, e.g., Clarke \cite[Theorem 2.7.5]{Clarke}) ensures $J$ is  Lipschitz continuous of constant $\|p\|_{L^2}$ and
$$
J^0(\by;\bz)\leq \int_{\Gamma_3^a} j^0_{,2}(\bx,\by;\bz)\; d\Gamma, \quad\forall \by,\bz\in Z_1.
$$
It follows immediately that
\begin{equation}\label{Aub-Cla_Est}
J^0(\gamma_1(\bu);\gamma_1(\bv-\bu))\leq \int_{\Gamma_3^a} j^0_{,2}(\bx, \bu_\tau;\bv_\tau-\bu_\tau)d\Gamma, \quad \forall \bu,\bv\in X.
\end{equation}
Thus, ${\bf (H_0')}$ holds with $\alpha_J:=\|p\|_{L^2}$ and $\alpha_L:=1$. Let  $\bar{w}_0:=\mathbf{0}_X$ and $\bar{\btau}_0\in Y$ be arbitrarily  fixed. Then
\begin{align*}
&\frac{2B(\bu,\bsigma)-B(\mathbf{0}_X,\bsigma)-B(\bu,\bar{\btau}_0)+\langle F(\bu),-\bu\rangle+\langle G(\bsigma),\bar{\btau}_0-\bsigma\rangle}{\|(\bu,\bsigma)\|}\\
&=\frac{B(\bu,\bsigma)-{\displaystyle \int_\Omega [\chi(\mathbf{0}_X)+\chi^\ast(\bar{\btau}_0)]\; dx-(\mathbf{f}_0,\bu)_{\mathcal{H}}-\int_{\Gamma_2}\mathbf{f}_2\cdot \bu d \Gamma}}{\sqrt{\|\bu\|_X^2+\|\bsigma\|_Y^2}}\\
&\geq \frac{C(\|\bu\|_X^2+\|\bsigma\|_Y^2)-c_1\|u\|_X-c_2}{\sqrt{\|\bu\|_X^2+\|\bsigma\|_Y^2}}\to\infty \mbox{ as }\|(\bu,\bsigma)\|\to\infty,
\end{align*}
for some suitable constants $c_1,c_2>0$. Therefore, the coercivity condition ${\bf (C)}$ also holds. Consequently, problem \eqref{eqn1} possesses at least one solution for any constraint sets $K\subseteq X $ and $\Lambda\subseteq Y$. 

We point out the fact $\Theta(\bw)\neq \emptyset$ for all $\bw\in \mathcal{K}_0$. In order to see this fix $\bw\in \mathcal{K}_0$, $\xi_0\in \partial_C J(\gamma_1(\bw))$ and define $\mathbf{f}\in X$ to be the unique element provided by the Riesz representation theorem such that 
$$
(\mathbf{f},\bv)_X:=(\mathbf{f}_0,\bv)_{\mathcal{H}}+\int_{\Gamma_2} \mathbf{f}_2 \cdot \bv d\Gamma,\ \forall \bv\in X.
$$ 
Let $\bmu_0:=\bvarepsilon(\mathbf{f}- \gamma_1^\ast(\xi_0))$, where $\gamma_1^\ast:Z_1\to X$ is the adjoint operator of $\gamma_1$. Then, $\mathbf{f}=\bvarepsilon^\ast(\bmu_0)+\gamma_1^\ast(\xi_0)$ and  for any $v\in\mathcal{K}_0$ one has
\begin{align*}
(\mathbf{f}_0,\bv)_{\mathcal{H}}+\int_{\Gamma_2} \mathbf{f}_2 \cdot \bv d\Gamma&=\left(\bvarepsilon^\ast(\bmu_0)+\gamma_1^\ast(\xi_0),\bv\right)_X\\
&=(\bmu_0,\bvarepsilon(\bv))_{\mathcal{H}}+\left(\xi_0,\gamma_1(\bv) \right)_{Z_1}\\
&\leq (\bmu_0,\bvarepsilon(\bv))_{\mathcal{H}}+ J^0(\gamma_1(\bw);\gamma_1(\bv))\\
& \leq (\bmu_0,\bvarepsilon(\bv))_{\mathcal{H}}+\int_{\Gamma_3^a} j^0_{,2}(\bx,\bw_{\tau};\bv_{\tau})\;d\Gamma
\end{align*}
which shows that $\bmu_0\in \Theta(\bw)$.

Choosing $K:=\mathcal{K}_0$ and $\Lambda :=Y$, Theorem \ref{C-Z_MainRes1}  ensures the existence of a pair  $(\bu_1,\bsigma_1)\in \mathcal{K}_0\times Y$ such that
\begin{equation}\label{IntermSol1}
\left\{
\begin{array}{l}
B(\bv,\bsigma_1)-B(\bu_1,\bsigma_1)+J^0(\gamma_1(\bu_1);\gamma_1(\bv-\bu_1))\geq \langle F(\bu_1),\bv-\bu_1\rangle, \ \forall v\in \mathcal{K}_0,\\
B(\bu_1,\bmu)-B(\bu_1,\bsigma_1)\geq 0,\ \forall \bmu\in Y.
\end{array}
\right.
\end{equation}
Now, applying again Theorem \ref{C-Z_MainRes1} with $K:=\mathcal{K}_0$ and $\Lambda:=\Theta(\bu_1)$ we infer there exist $\bu_2\in\mathcal{K}_2$ and $\bsigma_2\in\Theta(\bu_1)$ such that
\begin{equation}\label{IntermSol2}
\left\{
\begin{array}{l}
B(\bv,\bsigma_2)-B(\bu_2,\bsigma_2)+J^0(\gamma_1(\bu_2);\gamma_1(\bv-\bu_2))\geq \langle F(\bu_2),\bv-\bu_2\rangle, \ \forall v\in \mathcal{K}_0,\\
B(\bu_2,\bmu)-B(\bu_2,\bsigma_2)\geq 0,\ \forall \bmu\in \Theta(\bu_1).
\end{array}
\right.
\end{equation}
But, from the definition of the coupling function $B$ we have
$$
B(\bv,\bsigma_1)-B(\bu_1,\bsigma_1)=B(\bv,\bsigma_2)-B(\bu_1,\bsigma_2), \ \forall \bv\in X,
$$
and
$$
B(\bu_2,\bmu)-B(\bu_2,\bsigma_2)=B(\bu_1,\bmu)-B(\bu_1,\bsigma_2), \ \forall \bmu\in Y,
$$
which combined with \eqref{IntermSol1}-\eqref{IntermSol2} shows that $(\bu_1,\bsigma_2)$ solves the following system
\begin{equation}\label{IntermSol3}
\left\{
\begin{array}{l}
B(\bv,\bsigma_2)-B(\bu_1,\bsigma_2)+J^0(\gamma_1(\bu_1);\gamma_1(\bv-\bu_1))\geq \langle F(\bu_1),\bv-\bu_1\rangle, \ \forall v\in \mathcal{K}_0,\\
B(\bu_1,\bmu)-B(\bu_1,\bsigma_2)\geq 0,\ \forall \bmu\in \Theta(\bu_1).
\end{array}
\right.
\end{equation}
It is readily seen that, due to \eqref{Aub-Cla_Est}, the pair $(\bu_1,\bsigma_2)$ also solves \eqref{BipVarForm}, i.e., it is a weak solution for problem ${\bf (P)}$.
\end{proof}

\section*{Acknowledgment}
	 
	This project has received funding from the  Natural Science Foundation of Guangxi Grant Nos. GKAD21220144, 2021GXNSFFA196004 and GKAD23026237, the NNSF of China Grant Nos. 12001478, 12101143 and 12371312, the China Postdoctoral Science Foundation funded project No. 2022M721560,  the Startup Project of Doctor Scientific Research of Yulin Normal University No. G2023ZK13,  and the European Union's Horizon 2020 Research and Innovation Programme under the Marie Sklodowska-Curie grant agreement No. 823731 CONMECH. It is also supported by the project cooperation between Guangxi Normal University and Yulin Normal University. The second author is supported by the grant from the National Program for Research of the National Association of Technical Universities - GNAC ARUT 2023, ID: 220235047, "Sisteme cuplate de inegalit\u a\c ti hemivaria\c tionale \c si aplica\c tii".
	
	\section*{Data availability statement}
	Data sharing not applicable to this article as no data sets were generated or analysed during the current study.
	
	\section*{Ethical Approval}
	Not applicable.
	
	\section*{Competing interests} There is no conflict of interests.
	
	\section*{Authors' contributions} The authors contributed equally to this paper.

\pagestyle{plain}
\bibliographystyle{siam}

\bibliography{Biblio}

\end{document}